\documentclass[12pt]{amsart}
\usepackage{amsfonts,amsmath,amsthm,amssymb,enumitem,appendix}
\usepackage{tikz}
\usepackage{MnSymbol}
\usepackage[all]{xy}
\usepackage[margin=1.1in]{geometry}
\usepackage[draft = false]{hyperref}

\theoremstyle{plain}
\newtheorem{theorem}{Theorem}[section]
\newtheorem{thm}[theorem]{Theorem}
\newtheorem{lem}[theorem]{Lemma}
\newtheorem{lemma}[theorem]{Lemma}
\newtheorem{prop}[theorem]{Proposition}
\newtheorem{cor}[theorem]{Corollary}

\theoremstyle{remark}

\theoremstyle{definition}
\newtheorem{example}[theorem]{Example}
\newtheorem{defn}[theorem]{Definition}
\newtheorem{rmk}[theorem]{Remark}

\newcommand{\bbc}{\mathbb{C}}

\newcommand{\bbq}{\mathbb{Q}}

\newcommand{\bbt}{\mathbb{T}}

\newcommand{\bbz}{\mathbb{Z}}

\newcommand{\mca}{\mathcal{A}}

\newcommand{\mce}{\mathcal{E}}
\newcommand{\mcf}{\mathcal{F}}
\newcommand{\mcg}{\mathcal{G}}

\newcommand{\mcm}{\mathcal{M}}
\newcommand{\mcn}{\mathcal{N}}

\newcommand{\mcr}{\mathcal{R}}
\newcommand{\mcs}{\mathcal{S}}

\newcommand{\mcy}{\mathcal{Y}}

\newcommand{\mfs}{\mathfrak{S}}

\renewcommand{\aa}{\alpha}
\newcommand{\bb}{\beta}

\newcommand{\Lam}{\Lambda}

\newcommand{\ov}{\overline}

\newcommand{\sgn}{\textnormal{sign}}
\newcommand{\triv}{\textnormal{triv}}

\newcommand{\starr}{\textnormal{star}}
\newcommand{\CMG}{\textnormal{CMG}}

\newcommand{\bfp}{\mathbf{P}}
\newcommand{\bfx}{\mathbf{X}}
\newcommand{\bfy}{\mathbf{Y}}

\DeclareMathOperator{\Dih}{\textnormal{Dih}}

\DeclareMathOperator{\tGr}{\hat{\textnormal{Gr}}}
\DeclareMathOperator{\GL}{\textnormal{GL}}
\DeclareMathOperator{\SL}{\textnormal{SL}}

\DeclareMathOperator{\Span}{\textnormal{span}}
\DeclareMathOperator{\QAut}{\textnormal{QAut}}

\DeclareMathOperator{\Aut}{\textnormal{Aut}}

\title{Quasi-homomorphisms of cluster algebras}
\author{Chris Fraser}
\date{}
\keywords{Cluster algebra, seed orbit, quasi-homomorphism, cluster modular group, tagged mapping class group.}
\thanks{This work was supported by a graduate fellowship from the National Physical Science Consortium and NSF grant DMS-1361789. }
\subjclass[2010]{13F60}
\address{Department of Mathematics, University of Michigan, Ann Arbor, MI, 48109, USA}
\email{cmfra@umich.edu}
\begin{document}

%\thanks{This work was supported by a graduate fellowship from the National Physical Science Consortium and NSF grant DMS-1361789. }
%\subjclass[2010]{13F60}
\setcounter{tocdepth}{1}

\numberwithin{equation}{section}

\begin{abstract}
We introduce quasi-homomorphisms of cluster algebras, a flexible notion of a map between cluster algebras of the same type (but with different coefficients). The definition is given in terms of seed orbits, the smallest equivalence classes of seeds on which the mutation rules for non-normalized seeds are unambiguous. We present examples of quasi-homomorphisms involving familiar cluster algebras, such as cluster structures on Grassmannians, and those associated with marked surfaces with boundary. 
We explore the related notion of a quasi-automorphism, and compare the resulting group with other groups of symmetries of cluster structures. 
For cluster algebras from surfaces, we determine the subgroup of quasi-automorphisms inside the tagged mapping class group of the surface.
\end{abstract}

\maketitle
%\tableofcontents
\vspace{-.3in}

\section*{Introduction}
The general structural theory of cluster algebras has been well developed during the 15 years since their inception \cite{CAI}. Despite this, there does not seem to be a consensus on what the ``right'' notion of a homomorphism between cluster algebras should be-- several such notions have arisen in different mathematical settings, see e.g.~\cite{ADS, ASS, ChangZhua, ChangZhub, Reading, ReadingSurfaces}. From our perspective, the key difficulty in defining homomorphisms of cluster algebras is rooted in the fact that the construction of a cluster algebra involves three operations:  the addition, the multiplication, and the auxiliary addition used in the normalization condition. Most preexisting notions of a ``cluster homomorphism'' are designed to respect all three of these operations, a rather restrictive requirement. We suggest instead that even in the ordinary (i.e., normalized) setting, it is fruitful to consider maps that only preserve the structures intrinsic to non-normalized cluster algebras,  ignoring the auxiliary addition. This leads us to the concept of \emph{seed orbits} and to the mutation patterns these orbits form. The morphisms between such mutation patterns are the main object of our interest; we call them \emph{quasi-homomorphisms}. This paper is devoted to a systematic study of quasi-homomorphisms and related algebraic constructs. 

\begin{center}
\rule{6cm}{0.4pt}
\end{center}

A  cluster algebra is defined by specifying a distinguished set of generators (called \emph{cluster variables}) inside an ambient field of rational functions in $n$ variables. Starting from an \emph{initial cluster} of $n$ cluster variables, the remaining cluster variables are obtained by iterating algebraic steps called \emph{mutations}. Each mutation produces a new cluster from a current one by exchanging one cluster variable for a new one. The specific rules for computing the latter are encoded by two additional ingredients, an $n \times n$ \emph{exchange matrix}~$B$ and a \emph{coefficient tuple}~$\mathbf{p}$ consisting of elements of some fixed \emph{coefficient group}. The triple consisting of the cluster, the exchange matrix, and the coefficient tuple, is called a \emph{seed}. When a cluster mutates, the ingredients~$B$ and~$\mathbf{p}$ also do: the new matrix $B'$ is given explicitly in terms of $B$, and the new tuple $\mathbf{p}'$ satisfies a constraint involving~$\mathbf{p}$ and~$B$. A collection of seeds related to each other by mutations in all possible directions forms a \emph{seed pattern}. 

In the most general cluster algebra setup -- that of \emph{non-normalized seed patterns}~\cite{CAIII,CATSII, CAI}, the mutation recipe does not uniquely specify the new coefficient tuple~$\mathbf{p}'$ from~$\mathbf{p}$ and~$B$. This ambiguity propagates through iterated mutations, and consequently the set of cluster variables is not uniquely determined by the initial seed. 

The usual way to remove this ambiguity is to  impose the additional assumption that the coefficient group is endowed with an additional operation of ``auxiliary addition'' (making it into a \emph{semifield}), and then require the corresponding \emph{normalization condition} to hold at every seed. This assumption is satisfied for the most important examples of cluster algebras arising in representation theory. In this paper, we make use of another way of removing the ambiguity by considering \emph{seed orbits}, the smallest equivalence classes of seeds on which the mutation rules are unambiguous. This gives rise to to the concept of a mutation pattern of seed orbits. Such a pattern is determined uniquely by any one of its constituent seed orbits. The natural notion of a homomorphism between two mutation patterns of seed orbits brings us to the definition of a \emph{quasi-homomorphism}, a rational map (more precisely, a semifield homomorphism) that respects the seed orbit structure and commutes with mutations.  

Though the appropriate context for defining quasi-homomorphisms is that of non-normalized seed patterns, we see two ways in which quasi-homomorphisms are useful in the structural theory of ordinary (normalized) seed patterns. First, it is important to understand the relationships between cluster algebras with the same underlying pattern of exchange matrices but with different choices of coefficients. One celebrated result of this kind is the \emph{separation of additions formula} (\cite[Theorem $3.7$]{CAIV}). For a given mutation pattern of exchange matrices, this formula expresses the cluster variables in a cluster algebra with \emph{any} choice of coefficients in terms of those in a cluster algebra with a special choice of \emph{principal coefficients}. Our Proposition \ref{qhnormalizedprop} puts this formula in a wider context, in which every quasi-homomorphism between a pair of normalized seed patterns witnesses its own separation of additions. This idea can be used to construct a new cluster algebra starting from a known one. (More precisely, using a known cluster structure on an algebra~$R$, one can produce a cluster structure on another algebra~$R'$ by describing an appropriate map from~$R$ to~$R'$.)

Second, the naturally defined concept of a \emph{quasi-automorphism} gives rise to the \emph{quasi-automorphism group} of a seed pattern. This group interpolates between previously defined groups that are either  too sensitive to coefficients (these groups are too small) or don't refer to coefficients at all (these groups are too large). Most cluster algebras arising in applications have nontrivial coefficients, and these cluster algebras often afford nontrivial self-maps that are quasi-automorphisms of the cluster structure. The twist map on the Grassmannian~\cite{MarshScott} is one important example, cf.~Remark~\ref{twistremark}. In a forthcoming companion paper we construct a large group of quasi-automorphisms of the Grassmannian cluster algebras \cite{Scott} whose action on cluster variables has a simple description. Much of the abstract setup in this paper was developed with that application in mind. 

%\medskip
%A quasi-homomorphism is similar in spirit -- but different in details --  to several preexisting notions in the literature, namely those of \emph{cluster automorphism} and \emph{rooted cluster morphism} (cf.~\cite{ADS, ASS}) and \emph{coefficient specialization} (cf.~\cite{CAII, Reading, ReadingSurfaces}). We demarcate the key differences in Section \ref{QHSecn} once we have given the relevant definitions. 

%We would also like to mention that quasi-homomorphisms have been useful in a forthcoming book on cluster algebras \cite{CABook} to prove one direction of the finite type classification (namely, that a cluster algebra with a seed whose principal part is an orientation of a Dynkin quiver necessarily has only finitely many seeds).  

\medskip

The paper is organized as follows. Section \ref{PreliminariesSecn} presents background on non-normalized seed patterns. This is mostly standard and taken from \cite{CATSII, CAIV}, but with emphasis on the notion of the \emph{ambient semifield}, cf.~Definition~\ref{ambientsemifielddefn}. The section ends with a motivating example: a pair of seed patterns which will illustrate the various notions in subsequent sections. A reader familiar with cluster algebras can skim this section and head directly to Example \ref{GrandBandsExchangeGraphs}. 

Sections \ref{SeedOrbitsSecn} and \ref{QHSecn} are the conceptual core of the paper. We define seed orbits as the smallest equivalence classes of non-normalized seeds on which the mutation rule is unambiguous. In Proposition \ref{equivalentseedscondition} we give a more explicit characterization of seed orbits as orbits with respect to a rescaling action on seeds. Section \ref{QHSecn}  introduces quasi-homomorphisms of seed patterns and their basic properties. We end this section by describing the key differences between quasi-homomorphisms and some preexisting notions, specifically \emph{rooted cluster morphisms} \cite{ADS} and \emph{coefficient specializations} \cite{CAII, Reading, ReadingSurfaces}.

In Section \ref{NormalizedQHs} we discuss quasi-homomorphisms between normalized seed patterns. For seed patterns of geometric type, we relate quasi-homomorphisms to linear combinations of the rows of an extended exchange matrix, making connections to the separation of additions formula and to gradings on cluster algebras. Section \ref{NervesSecn} introduces the easiest way of specifying a quasi-homomorphism in practice, by checking that a given semifield map sends cluster variables to rescaled cluster variables on a \emph{nerve}. In Section \ref{QASecn} we define the \emph{quasi-automorphism group} of a seed pattern and compare it with the cluster modular group \cite{FG} and the group of cluster automorphisms \cite{ASS}.

Sections \ref{SurfacesSecn} and \ref{SurfacesProofsSecn} focus on cluster algebras associated with bordered marked surfaces~\cite{CATSI, CATSII}. The main result is Theorem \ref{whichgworkwithL} describing the quasi-automorphism group of such a cluster algebra as a subgroup of the tagged mapping class group (excluding a few exceptional surfaces). In particular, it establishes that regardless of the choice of coefficients in such a cluster algebra, the quasi-automorphism group is always a finite index subgroup of the cluster modular group.

%Sections \ref{SurfacesSecn} and \ref{SurfacesProofsSecn} focus on cluster algebras associated with bordered marked surfaces~\cite{CATSI, CATSII}. The main result is Theorem \ref{whichgworkwithL} describing the tagged mapping classes that represent quasi-automorphisms of such a cluster algebra. In particular, it establishes that regardless of the choice of coefficients in such a cluster algebra, the set of such tagged mapping classes always forms a finite index subgroup inside the tagged mapping class group. (which conjecturally coincides with the \emph{cluster modular group} of Fock and Goncharov). 
%We interpret this result as evidence that quasi-automorphisms give a ``rich'' group of automorphisms of cluster algebras in the presence of nontrivial coefficients. 

The concept of a nerve introduced in Section \ref{NervesSecn} is new and includes as a special case the star neighborhood of a vertex. Star neighborhoods show up in the algebraic Hartogs' principle argument used to establish that a given cluster algebra is contained in another algebra \cite[Proposition 3.6]{tensors}.  In Appendix \ref{StarfishSecn} we extend this argument from a star neighborhood to an arbitrary nerve. 

Section \ref{GrassmannianandBands} illustrates the techniques in Section \ref{NormalizedQHs}. We generalize Example \ref{GrandBandsExchangeGraphs} by describing a quasi-isomorphism between the Grassmannian cluster algebras \cite{Scott} and polynomial rings arising as coordinate rings of \emph{band matrices}.

\section*{Acknowledgements}
I would like to thank Ian Le, Greg Muller, and Gregg Musiker for helpful conversations. I especially thank Sergey Fomin for many conversations and suggestions. This work was supported by a graduate fellowship from the National Physical Science Consortium and NSF grant DMS-1361789. While in the midst of carrying out this work, I learned that Thomas Lam and David Speyer had independently obtained results similar to Corollary \ref{geomtypeseedhom} and Remark \ref{separationofadditionsrmk}.

%\addcontentsline{toc}{section}{Acknowledgement}

%%%%%%%%%%%%%%%%%%%%%%%%%%%%%%%%%%%%%%%%%%%%%%%%%%%
\section{Preliminaries on seed patterns}\label{PreliminariesSecn}
A (non-normalized) cluster algebra is constructed from a set of data called a \emph{non-normalized seed pattern}. We define this data now while fixing standard notation. For a number $x$ we let $[x]_+ := \max (x,0)$. We let $\sgn(x)$ equal either  $-1$, $0$ or $1$ according to whether $x$ is negative, zero, or positive. We denote $\{1,\dots,n\}$ by $[1,n]$. 

The setup %of a non-normalized seed pattern 
begins with a choice of \emph{ambient field of rational functions} $\mcf$ with coefficients in a \emph{coefficient group } $\bfp$. The coefficient group is an abelian multiplicative group without torsion. The ambient field is a field of rational functions in $n$ variables with coefficients in $\bfp$: it is the set of expressions that can be made out of $n$ elements $x_1,\dots,x_n$ and the elements of of $\bfp$, using the standard arithmetic operations $+,-,\times$ and $\div$, under the usual notion of equivalence of such rational expressions. The integer $n$ is called the \emph{rank}.

\begin{defn}[Non-normalized seed, \cite{CAI,CATSII}]\label{nnseeddefn}
Let $\bfp$ and $\mcf$ be as above. A \emph{non-normalized seed} in $\mcf$ is a triple $\Sigma = (B,\mathbf{p},\mathbf{x})$, consisting of the following three ingredients: 
\begin{itemize}
\item a skew-symmetrizable $n \times n$ matrix $B = (b_{ij})$,
\item a \emph{coefficient tuple} $\mathbf{p} = (p_1^{\pm},\dots,p_n^{\pm})$ consisting of $2n$ elements in $\bfp$, 
\item a \emph{cluster} $\mathbf{x} = (x_1,\dots,x_n )$ in $ \mcf$, whose elements (called \emph{cluster variables}) 
are algebraically independent and freely generate  $\mcf$ over $\bbq \bfp$. 
\end{itemize}
\end{defn}  

The more restrictive notion of \emph{normalized seed} is given in Definition \ref{normalizedpatterndefn}. Normalized seeds are much more studied in the literature, where they are usually simply called \emph{seeds}. Thus, we persistently use the adjective \emph{non-normalized} in our setting, although this is a little clumsy.

\begin{defn} A \emph{labeled $n$-regular tree}, $\bbt_n$, is an $n$-regular tree with edges labeled by integers so that the set of labels emanating from each vertex is $[1,n]$. We write $t \xrightarrow{k} t'$ to indicate that vertices $t,t'$ are joined by an edge with label $k$. An isomorphism $\bbt_n \to \ov{\bbt}_n$ of labeled trees $\bbt_n$ and $\ov{\bbt}_n$ sends vertices to vertices and edges to edges, preserving incidences of edges and the edge labels. Such an isomorhism is uniquely determined by its value at a single vertex $t \in \bbt_n$. 
\end{defn}

\begin{defn}[Non-normalized seed pattern, \cite{CATSII,CAI}]\label{nnseedpatterndefn}
Let $\bfp$ and $\mcf$ be as above. %Fix a coefficient group $\bfp$ and an ambient field $\mcf$ with coefficients in $\bfp$. 
A collection of non-normalized seeds in~$\mcf$, with one seed~$\Sigma(t)= (B(t),\mathbf{p}(t),\mathbf{x}(t))$ for each~$t \in \bbt_n$, is called a \emph{non-normalized seed pattern} if for each edge~$t \xrightarrow{k} t'$, the seeds $\Sigma(t)$ and $\Sigma(t')$ are related by a \emph{mutation in direction~$k$}:
\begin{itemize}
\item The matrices $B(t)$ and $B(t')$ are related by a \emph{matrix mutation}
\begin{equation}\label{brecurrence}
b_{ij}(t') = 
\begin{cases}
-b_{ij}(t) & \text{ if } i = k \text{ or } j = k \\
b_{ij}(t)+ \sgn(b_{ik}(t))[b_{ik}(t)b_{kj}(t)]_+ & \text{ otherwise,}
\end{cases}
\end{equation}
\item  the coefficient tuples $\mathbf{p}(t)$ and $\mathbf{p}(t')$ are related by  
\begin{equation}\label{nnyrecurrenceatk}
p^\pm_k(t') = p^\mp_k(t) \text{ and }
\end{equation} 
\begin{equation}\label{nnyrecurrence}
\frac{p^+_j(t')}{p^-_j(t')} = 
\begin{cases}
\frac{p^+_j(t)}{p^-_j(t}  p^{+}_k(t)^{b_{kj(t)}} & \text{ if } b_{kj} \geq 0 \\
\frac{p^+_j(t)}{p^-_j(t)}  p^{-}_k(t)^{b_{kj(t)}} & \text{ if } b_{kj} \leq 0 \\
\end{cases}
\end{equation}
when $j \neq k$, 
\item and the clusters $\mathbf{x}(t)$ and $\mathbf{x}(t')$ are related by 
\begin{equation}\label{nnxrecurrenceatj}
x_j(t') = x_j(t) \text{ for $j \neq k$, and }
\end{equation} 
\begin{equation}\label{nnxrecurrence}
x_k(t)x_k(t') =p^+_k \prod x_j(t)^{[b_{jk}]_+} +  p^-_k \prod x_j(t)^{[-b_{jk}]_+} ,
\end{equation}
the latter of which is called an \emph{exchange relation}. 
\end{itemize}
\end{defn}

The rules \eqref{brecurrence} through \eqref{nnxrecurrence} are ambiguous, meaning $\Sigma(t')$ is not determined uniquely from $\Sigma(t)$. Indeed, since \eqref{nnyrecurrence} only mentions the ratio $\frac{p^+_j(t')}{p^-_j(t')}$, for each $j \neq k $ one can rescale both of  $p^+_j(t')$ and $p^-_j(t')$ by a common element of $\bfp$ while preserving \eqref{nnyrecurrence}. We write~$\Sigma \overset{\mu_k}{\leftrightsquigarrow} \Sigma'$ to indicate that two seeds~$\Sigma$ and~$\Sigma$ are related by a mutation in direction~$k$; this condition is symmetric in~$\Sigma$ and~$\Sigma'$.   

Thinking of \eqref{nnxrecurrence} as a recipe for computing $x_k(t')$ from $\Sigma(t)$, we crucially observe that the computation is \emph{subtraction-free}: the only operations needed are $+,\times$ and $\div$ in $\mcf$. This motivates the following definition: 

\begin{defn}[Ambient semifield]\label{ambientsemifielddefn}
Let $\mce$ be a non-normalized seed pattern, and $\mathbf{x}(t)$ one of its clusters. The \emph{ambient semifield}, $\mcf_{>0} = \mcf_{>0}(\mce) \subset \mcf$ is the subset of all elements which can be given as a \emph{subtraction-free} rational expression in the elements of $\mathbf{x}(t)$, with coefficients in $\bfp$. Thus, it is the set of rational functions which can be built out of~$x_1(t),\dots,x_n(t)$ and the elements of $\bfp$ using the operations $+,\times$ and $\div$ in $\mcf$. 
\end{defn}

Since \eqref{nnxrecurrence} is subtraction-free, $\mcf_{>0}$ is independent of the choice of $t$ (it only depends on $\mce$), and every cluster variable for $\mce$ lies in $\mcf_{>0}$. Recall that a \emph{semifield} is an abelian multiplicative group, with an additional binary operation (called the auxiliary addition) that is commutative and associative, and distributes over multiplication. The ambient semifield is a semifield with respect to the multiplication and addition operations in $\mcf$, justifying its name. Homomorphisms between semifields are defined in the obvious way. The ambient semifield has the following universality property.

\begin{lem}[{\cite[Definition $2.1$]{CAIV}}]\label{universalsemifieldlemma}
Let $\mce$ be a non-normalized seed pattern with coefficient group $\bfp$ and ambient semifield $\mcf_{>0}$.  Fix a cluster $\mathbf{x}(t)$ in $\mce$. Let $\mcs$ be \emph{any} semifield. Then given a multiplicative group homomorphism $\bfp \to \mcs$, and a function $\mathbf{x}(t) \to \mcs$, there exists a unique semifield homomorphism $\mcf_{>0} \to \mcs$ agreeing with the given maps on $\bfp \cup \mathbf{x}(t)$. 
\end{lem}

The following elements of $\mcf_{>0}$ will play a prominent role in Section \ref{SeedOrbitsSecn}. 
\begin{defn}[Hatted variables]\label{yhatsdefn}
Let $\mce$ be a non-normalized seed pattern. Let $\hat{\mathbf{y}}(t) = (\hat{y}_1(t),\dots,\hat{y}_n(t))$ denote the $n$-tuple of \emph{hatted variables}   
\begin{equation}\label{hattedvarsdefn}
\hat{y}_j(t) = \frac{p^+_j(t)}{p^-_j(t)} \prod_{i} x_i(t)^{b_{ij}(t)},
\end{equation} 
obtained by taking the ratio of the two terms on the right hand side of \eqref{nnxrecurrence}.
\end{defn}

%Thus the quantity $\hat{y}_k(t)$ is the ratio of the two terms on the right hand side of an exchange relation \eqref{nnxrecurrence}. 
The hatted variables in adjacent seeds determine each other as follows:  
\begin{prop}[{\cite[Proposition $2.9$]{CATSII}}] \label{nnyhatsmutate} 
Let $\mce = (B(t),\mathbf{p}(t),\mathbf{x}(t))$ be a non-normalized seed pattern with hatted variables $\hat{\mathbf{y}}(t)$. For each edge $t \xrightarrow{k} t'$, the $n$-tuples $\hat{\mathbf{y}}(t)$ and $\hat{\mathbf{y}}(t')$ satisfy
\begin{equation}\label{yrecurrencewplus}
\hat{y}_{j}(t') = 
\begin{cases}
\hat{y}_j(t)^{-1} & \text{ if } j = k  \\
\hat{y}_j(t) \hat{y}_k(t)^{[b_{kj}(t)]_+}(\hat{y}_k(t) + 1)^{-b_{kj}(t)} & \text{ if } j \neq k. 
\end{cases}
\end{equation}
\end{prop}

The propagation rule \eqref{yrecurrencewplus} takes place in $\mcf_{>0}$, and only depends on the $B$ matrix. 

%%%%%%%%%%%%%%%%%%%%%%%%%%%%%%%%%%%%%%%%%%

\medskip

The preceding discussion is what we will need for Section \ref{SeedOrbitsSecn}. We briefly recall a few more definitions which will be useful in presenting our examples. First, the \emph{exchange graph} $\mathbf{E}$ associated with a seed pattern $\mce$ is the graph whose vertices are the \emph{unlabeled seeds} in $\mce$, and whose edges correspond to mutations between these seeds. More precisely, permuting the indices $[1,n]$ in a non-normalized seed commutes with the mutation rules \eqref{brecurrence} through \eqref{nnxrecurrence}. The exchange graph is the $n$-regular graph obtained by identifying vertices $t_1,t_2 \in \bbt_n$ if the  seeds $\Sigma(t_1)$ and  $\Sigma(t_2)$ are permutations of each other. The \emph{star neighborhood} $\starr(t)$ of a vertex $t \in \mathbf{E}$ is the set of $n$ edges adjacent to it. Rather than being indexed by $[1,n]$, the data in an unlabeled seed~$\Sigma(t)$ for $t \in \mathbf{E}$ is indexed by the $n$ seeds adjacent to~$\Sigma(t)$, i.e. by the elements of $\starr(t)$.  

Second, in the concrete examples in this paper, we have chosen a distinguished finite set of elements called \emph{frozen variables}, and the coefficient group $\bfp$ is the free abelian multiplicative group of Laurent monomials in these frozen variables. The \emph{cluster algebra}~$\mca$ associated with the seed pattern~$\mce$ is the~$\bbz$-algebra generated by the frozen variables and all  of the cluster variables arising in the seeds of~$\mce$.

\begin{example}\label{GrandBandsExchangeGraphs}
We now introduce a pair of affine algebraic varieties $\bfx$ and $\bfy$ and a pair of seed patterns in their respective fields of rational functions. The cluster algebras associated with these seed patterns are the coordinate rings $\bbc[\bfx]$ and $\bbc[\bfy]$. Both cluster algebras are of finite Dynkin type $A_2$.

Let $\bfx = \tGr(3,5)$ be the affine cone over the Grassmann manifold of $3$-dimensional planes in $\bbc^5$. The points in $\bfx$ are the decomposable tensors $\{x \wedge y \wedge z \colon x,y,z \in \bbc^5\} \subset \Lambda^3(\bbc^5)$. Its coordinate ring is generated by the \emph{Pl\"ucker coordinates} $\Delta_{ijk}$ for $1 \leq i < j < k \leq 5$, extracting the coefficient of $e_i \wedge e_j \wedge e_k$ in $x \wedge y \wedge z$, where $e_1,\dots,e_5$ is the standard basis for $\bbc^5$. 
%Thus, $\Delta_{ijk}$ is the maximal minor occupying columns $i,j$ and $k$ of a generically chosen $3 \times 5$ matrix. 
Representing a given $x \in \tGr(3,5)$ by a $3 \times 5$ matrix, $\Delta_{ijk}(x)$ is the maximal minor of this matrix in columns $i,j,$ and $k$. 

There is a well known cluster structure on $\bbc[\bfx]$ \cite{CAI,CAII}. It is a special case of a cluster structure for arbitrary Grassmannians constructed by Scott \cite{Scott}. The frozen variables are the Pl\"ucker coordinates consisting of cyclically consecutive columns
\begin{equation}
\Delta_{123}, \Delta_{234}, \Delta_{345}, \Delta_{145}, \Delta_{125} \label{FrozenXs}. \\   
\end{equation}

There are five cluster variables, listed in \eqref{ClusterXs} with cyclically adjacent pairs of cluster variables forming clusters 
\begin{equation}
\Delta_{245}, \Delta_{235}, \Delta_{135}, \Delta_{134}, \Delta_{124}.  \label{ClusterXs}   
\end{equation}
The clusters and exchange relations are given in Figure \ref{GrThreeFiveFig}. All of the other data in the seed pattern can be determined from these. For example, focusing on the seed whose cluster is $(x_1,x_2) = (\Delta_{235}, \Delta_{245})$, from the first and fifth exchange relations in Figure \ref{GrThreeFiveFig} follows
\begin{align}
(p_1^+,p^-_1,p_2^+,p_2^-) &= (\Delta_{125}\Delta_{234},\Delta_{123},\Delta_{145},\Delta_{345}\Delta_{125}) \\
(\hat{y}_1, \hat{y}_2) &=(\displaystyle{\frac{\Delta_{125}\Delta_{234}}{\Delta_{123}\Delta_{245}}, \frac{\Delta_{145}\Delta_{235}}{\Delta_{345}\Delta_{125}}}).
\end{align}
The exchange relations are written so that mutating is moving clockwise in the exchange graph. If a mutation moves counterclockwise, one should swap the order of the two terms in the exchange relation.

\begin{figure}[ht]
%\begin{multicols}{2}
\begin{tabular}{ll}
\begin{tikzpicture}[scale = .5]
\def  \rsize{3.25}; 

\foreach \s in {1,...,5}
{
 \node at ({72*(-\s)+162}:\rsize) {$\bullet$};
 \draw ({72*(-\s)+162}:\rsize) -- ({72*(-\s+1)+162}:\rsize);
}

\node at (90: 1.2*\rsize) {$ (\Delta_{235},\Delta_{245})$};
\node at (25 :1.5*\rsize) {$(\Delta_{135},\Delta_{235})$};
\node at (-52: 1.3*\rsize) {$(\Delta_{134}, \Delta_{135})$};
\node at (-140: 1.5*\rsize) {$(\Delta_{124}, \Delta_{134})$};
\node at (-210: 1.4*\rsize) {$(\Delta_{245}, \Delta_{124})$};
\end{tikzpicture}
& 
{
$\begin{aligned}
\Delta_{245} \Delta_{135} &= \Delta_{145} \Delta_{235}+\Delta_{125}\Delta_{345} \\[.1cm] 
\Delta_{235} \Delta_{134} &= \Delta_{234} \Delta_{135}+\Delta_{123}\Delta_{345} \\[.1cm] 
\Delta_{135} \Delta_{124} &= \Delta_{125} \Delta_{134}+\Delta_{123}\Delta_{145} \\[.1cm] 
\Delta_{134} \Delta_{245} &= \Delta_{345} \Delta_{124}+\Delta_{123}\Delta_{345} \\[.1cm] 
\Delta_{124} \Delta_{235} &= \Delta_{123} \Delta_{245}+\Delta_{125}\Delta_{234}
\end{aligned}
$
}

\end{tabular}
%\end{multicols}
\caption{The exchange graph for $\bbc[\bfx]$. The vertices are clusters and edges between vertices are mutations. Each mutation exchanges two cluster variables via an exchange relation listed in the table at top right. The extra data in each seed can be inferred from these exchange relations. 
\label{GrThreeFiveFig}}
\end{figure}
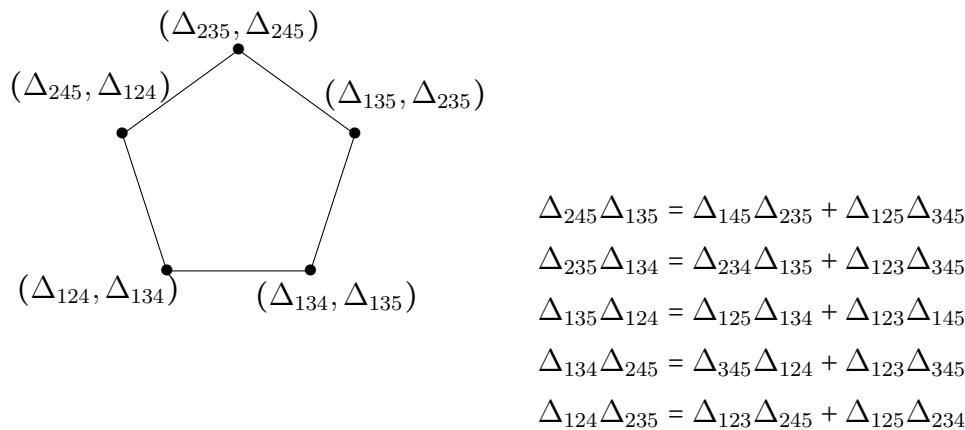

Second, let $\bfy \cong \bbc^9$ the affine space of \emph{band matrices} of the form
\begin{equation}
y =  \begin{pmatrix}
y_{1,1} &y_{1,2} & y_{1,3}& 0& 0 \\
0 & y_{2,2} & y_{2,3} & y_{2,4}& 0 \\
0 &0 & y_{3,3} &  y_{3,4} & y_{3,5}
\end{pmatrix}.\end{equation}
Its coordinate ring $\bbc[\bfy]$ contains the \emph{minors} $Y_{I,J}$. Evaluating $Y_{I,J}$ on $y \in \bfy$ returns the minor of $y$ occupying rows $I$ and columns $J$, e.g. $Y_{i,j}(y) = y_{i,j}$ and $Y_{12,23}(y) = y_{1,2}y_{2,3}-y_{1,3}y_{2,2}$. Some of these minors factor, e.g. $Y_{12,13} = Y_{1,1}Y_{2,3}$. 

Figure \ref{BandThreeFiveFig} shows a seed pattern whose cluster algebra is $\bbc[\bfy]$. The frozen variables are the following minors
\begin{equation}
Y_{1,1}, Y_{2,2}, Y_{3,3} , Y_{1,3},Y_{2,4},Y_{3,5}, Y_{123,234} \label{FrozenYs}. \\
\end{equation}

The cluster variables are listed in \eqref{ClusterYs}, with cyclically adjacent pairs forming clusters
\begin{equation}
Y_{1,2}, Y_{12,23},Y_{2,3}, Y_{23,34}, Y_{3,4}\label{ClusterYs}.     
\end{equation}  

%As can be checked by hand, this is a cluster algebra of type $A_2$: the frozen variables are in \eqref{FrozenYs}, and the cluster variables are in \eqref{ClusterYs}, with cyclically adjacent pairs forming clusters. Figure \ref{BandThreeFiveFig} gives the exchange graph and exchange relations, mirroring Figure \ref{GrThreeFiveFig}.  

\begin{figure}[ht]
\begin{center}
\begin{tabular}{ll}
\begin{tikzpicture}[scale = .5]
\def  \rsize{3.4}; 
\foreach \s in {1,...,5}
{
 \node at ({72*(-\s)+162}:\rsize) {$\bullet$};
 \draw ({72*(-\s)+162}:\rsize) -- ({72*(-\s+1)+162}:\rsize);
}

\node at (95: 1.2*\rsize) {$(Y_{12,23}, Y_{1,2})$};
\node at (30: 1.35*\rsize) {$(Y_{2,3}, Y_{12,23})$};
\node at (-60: 1.2*\rsize) {$(Y_{23,34}, Y_{2,3})$};
\node at (-140: 1.6*\rsize) {$(Y_{3,4}, Y_{23,34})$};
\node at (-205: 1.4*\rsize) {$(Y_{1,2}, Y_{3,4})$};
\end{tikzpicture}

&
{$
\begin{aligned}
Y_{1,2}Y_{2,3}     &= Y_{12,23}+Y_{2,2}Y_{1,3}  \\[.1cm]
Y_{12,23}Y_{23,34} &= Y_{123,234} Y_{2,3}+Y_{2,2}Y_{3,3}Y_{1,3}Y_{2,4} \\[.1cm]
Y_{2,3}Y_{3,4}     &= Y_{23,34}+Y_{3,3}Y_{2,4} \\[.1cm]
Y_{23,34}Y_{1,2}   &= Y_{2,2}Y_{1,3}  Y_{3,4}+Y_{123,234} \\[.1cm]
Y_{3,4}Y_{12,23}   &= Y_{3,3}Y_{2,4} Y_{1,2}+Y_{123,234} 
\end{aligned}
$
}
\end{tabular}
\caption{The exchange graph and exchange relations for $\bbc[\bfy]$, mirroring Figure \ref{GrThreeFiveFig}. \label{BandThreeFiveFig}}
\end{center}
\end{figure}
\end{example}

\section{Seed orbits}\label{SeedOrbitsSecn}
We introduce seed orbits by first describing them as equivalence classes under a certain equivalence relation on seeds. Proposition \ref{equivalentseedscondition} gives another characterization as orbits under an explicit rescaling action. 

\begin{defn}
Let $\vec{k} = (k_1,\dots,k_\ell)$ be a sequence of elements of $[1,n]$. Choosing a base point $t_0 \in \bbt_n$, such a sequence determines a walk $t_0 \xrightarrow {k_1} t_1 \xrightarrow{k_2} \cdots \xrightarrow{k_\ell} t_\ell $ in $\bbt_n$. We say that~$\vec{k}$ is \emph{contractible} if this walk starts and ends at the same vertex of $\bbt_n$, i.e. $t_\ell = t_0$. 

Given non-normalized seeds $\Sigma$ and $\Sigma^*$, we write $\Sigma \sim \Sigma^*$ if there is a contractible sequence of mutations from $\Sigma$ to $\Sigma^*$, i.e. a contractible sequence $\vec{k}$ and non-normalized seeds $\Sigma_1,\dots,\Sigma_{\ell-1}$ such that  
\begin{equation}\label{equivalentseedsequence}
\Sigma = \Sigma_0 \overset{\mu_{k_1}}{\leftrightsquigarrow} \Sigma_1 \overset{\mu_{k_2}}{\leftrightsquigarrow} \Sigma_2 \cdots \Sigma_{\ell-1} \overset{\mu_{k_\ell}}{\leftrightsquigarrow} \Sigma_{\ell} = \Sigma^*.
\end{equation}
\end{defn}

Clearly, $\sim$ is an equivalence relation on non-normalized seeds. Furthermore, it removes the ambiguity present in mutation of non-normalized seeds:  
\begin{lem}\label{equivalenceclassesmutate}
The mutation rule $\overset{\mu_k}{\leftrightsquigarrow}$ becomes unambiguous and involutive once it is thought of as a rule on equivalence classes of seeds under $\sim$. That is, fixing a $\sim$-equivalence class $\mfs$ and a direction $k \in [n]$, the set of seeds 
\begin{equation}
\{ \Sigma' \colon  \Sigma' \overset{\mu_k}{\leftrightsquigarrow} \Sigma \text{ for some } \Sigma \in \mfs \} 
 \end{equation}
is again a $\sim$-equivalence class of seeds. 
\end{lem}

We now characterize $\sim$-equivalence classes explicitly. We say two elements $z,x \in \mcf$ are \emph{proportional}, written $z \asymp x$, if  $\frac{z}{x} \in \bfp$. We emphasize that $\bfp$ does not include constants, e.g. $-1,2 \notin \bfp$, and thus $x$ is not proportional to $-x, 2x$, etc. 

\begin{prop}[Seed orbits]\label{equivalentseedscondition} Let $\Sigma = (B,\mathbf{p},\mathbf{x}), \Sigma^* = (B^*,\mathbf{p}^*,\mathbf{x}^*)$ be non-normalized seeds in $\mcf$, of rank $n \geq 2$, with $\mathbf{x} = (x_i),\mathbf{p} = (p^\pm_i), \mathbf{x}^* = (x^*_i),\mathbf{p}^* = ((p^*)^\pm_i)$. 
%\footnote{when $n=1$, non-normalized mutation is unambiguous}.
Then the following are equivalent: 
\begin{enumerate}
\item $\Sigma \sim \Sigma^*$.
\item $B = B^*, \hat{\mathbf{y}}(\Sigma) = \hat{\mathbf{y}}(\Sigma^*)$, and $x_i \asymp x^*_i$ for all $i$.  
\item $B = B^*$, and there exist scalars $c_1,\dots,c_n,d_1,\dots,d_n \in \bfp$, such that 
\begin{align}
x^*_j &= \frac{x_j}{c_j} \label{rescalingxs} \\
(p^*)^\pm_j &= \frac{p^{\pm}_j}{d_j} \prod c_i^{[\pm b_{ij}]_+}  \label{rescalingps}.
\end{align}
\end{enumerate}
\end{prop}

Equations \eqref{rescalingxs} and \eqref{rescalingps} define a \emph{rescaling action} of $\bfp^n \times \bfp^n$ on non-normalized seeds, denoted by $(\vec{c},\vec{d}) \cdot \Sigma$ where $(\vec{c},\vec{d}) \in \bfp^n \times \bfp^n$ and $\Sigma$ is a non-normalized seed. Proposition \ref{equivalentseedscondition} says that a $\sim$-equivalence class of non-normalized seeds is precisely a $\bfp^n \times \bfp^n$ orbit under this action; we henceforth refer to these equivalence classes as \emph{seed orbits}.

\begin{proof} 
Conditions (2) and (3) are a re-translation of each other by immediate calculation.

We show (1) implies (2). Defining a seed orbit by $(2)$, this implication follows from the fact that seed orbits are ``closed under mutation.'' More precisely, if $\Sigma$ and $\Sigma^\dagger = (\vec{c},\vec{d}) \cdot \Sigma$ are in the same seed orbit and $\Sigma'$ and  $(\Sigma^\dagger)'$ are two seeds satisfying  $\Sigma \overset{\mu_{k}}{\leftrightsquigarrow} \Sigma'$ and $\Sigma^\dagger \overset{\mu_{k}}{\leftrightsquigarrow} (\Sigma^\dagger)'$, then $\Sigma'$ and $(\Sigma^\dagger)'$ are in the same seed orbit. By \eqref{brecurrence} and Proposition \ref{nnyhatsmutate}, we know that $B' = (B^\dagger)'$ and $\hat{\mathbf{y}}' = (\hat{\mathbf{y}}^\dagger)'$, so the claim will follow if we check $(x^\dagger)'_j \asymp x'_j$ for all $j$. This is obvious when $j \neq k$ from \eqref{nnxrecurrenceatj}. When $j= k$, \eqref{nnxrecurrence} for the mutation $\Sigma^\dagger \overset{\mu_{k}}{\leftrightsquigarrow} (\Sigma^\dagger)'$ says that 
\begin{align}\label{nnxrecurrenceredone}
(x^\dagger)'_k &= (x_k^\dagger)^{-1}((p^\dagger)^+_k \prod (x^\dagger_j)^{[b_{jk}]_+} +  (p^\dagger)^-_k \prod (x^\dagger_j)^{[-b_{jk}]_+}) \\
&=(x_k^\dagger)^{-1}(p^\dagger)^-_k(1+\hat{y}_k(\Sigma^\dagger)) \prod (x^\dagger_j)^{[-b_{jk}]_+} \\
&= \frac{c_k} {d_k} (x_k)^{-1}  p^-_k(1+\hat{y}_k(\Sigma)) \prod x_j^{[-b_{jk}]_+} \\ 
&= \frac{c_k} {d_k} x_k', \label{nnxrecurrenceredoneend}
\end{align}
as desired. Returning to the implication $(1) \Rightarrow (2)$, from the symmetry of $\overset{\mu_k}{\leftrightsquigarrow}$ it follows that $\Sigma$ is related to itself along any contractible sequence $\vec{k}$. Since seed orbits are closed under mutation, any seed $\Sigma^*$ related to $\Sigma$ by a contractible sequence of mutations is therefore in the same seed orbit as $\Sigma$.   

Now we show $(3)$ implies $(1)$. Let $\hat{c}_j(a) \in \bfp^n \times \bfp^n$ denote the vector with $c_j = a$ and all other entries equal to $1$, and define similarly $\hat{d}_j(a)$. Clearly, it suffices to show that $\Sigma \sim \hat{c}_j(a) \cdot \Sigma$ and $\Sigma \sim \hat{d}_j(a) \cdot \Sigma$, since rescalings of this type generate $\bfp^n \times \bfp^n$. 

Seeds of the form $\hat{d}_j(a) \cdot \Sigma$ are equivalent to $\Sigma$, as follows by mutating twice in any direction $k \neq j$. For seeds of the form $\hat{c}_j(a) \cdot \Sigma$, let $\Sigma'$ be any seed satisfying $\Sigma \overset{\mu_{j}}{\leftrightsquigarrow}  \Sigma'$:
\begin{align}
\Sigma &\sim \hat{d}_j(a^{-1}) \cdot  \Sigma  \label{commutecalci}  \\
 &\overset{\mu_{j}}{\leftrightsquigarrow} (\hat{c}_j(a^{-1})\hat{d}_j(a^{-1})) \cdot \Sigma'  \label{commutecalcii}\\
&\sim (\hat{c}_j(a^{-1}) ) \cdot \Sigma'  \label{commutecalciii} \\
&\overset{\mu_{j}}{\leftrightsquigarrow} \hat{c}_j(a) \cdot \Sigma, \label{commutecalciv}
\end{align}
where \eqref{commutecalcii} and \eqref{commutecalciv} follow from the calculation in \eqref{nnxrecurrenceredoneend}, and \eqref{commutecalci} and \eqref{commutecalciii} are admissible since we already know rescaling by $\hat{d}_j(a)$ preserves equivalence of seeds. Since \eqref{commutecalci} through \eqref{commutecalciv} amounts to mutating in direction $j$ twice on seed equivalence classes, it follows that $\Sigma \sim \hat{c}_j( a) \cdot \Sigma$ as desired. 
\end{proof}

\section{Quasi-homomorphisms}\label{QHSecn}
We will now give the definition of a quasi-homomorphisms from a seed pattern $\mce$ to another seed pattern $\ov{\mce}$. We retain the notation of Section \ref{PreliminariesSecn} for all the data in $\mce$, and 
we use bars to denote the analogous quantities in the second pattern $\ov{\mce}$. Thus $\ov{\mce}$ has coefficient group $\ov{\bfp}$, ambient field  $\ov{\mcf}$, seeds $\ov{\Sigma}(\ov{t}) = (\ov{B}(\ov{t}),\ov{\mathbf{p}}(\ov{t}),\ov{\mathbf{x}}(\ov{t}))$, hatted variables $\hat{\ov{y}}_j(\ov{t})$, and so on. It is built on a second copy of the $n$-regular tree, $\ov{\bbt}_n$. 

The motivating observation is the following: since the mutation rules \eqref{nnyrecurrence} through \eqref{nnxrecurrence} are certain algebraic relations in in $\mcf_{>0}$, they are preserved by a homomorphism of semifields. 

\begin{defn}[Quasi-homomorphism]\label{qhdefn} Let $\mce$ and $\ov{\mce}$ be non-normalized seed patterns. Let~$\Psi \colon \mcf_{>0} \to \ov{\mcf}_{>0}$ be a semifield homomorphism satisfying~$\Psi(\bfp) \subset \ov{\bfp}$ (in this case we say~$\Psi$ \emph{preserves coefficients}). We say~$\Psi$  
is a \emph{quasi-homomorphism from $\mce$ to $\ov{\mce}$} if it maps each seed in $\mce$ to a seed that is $\sim$-equivalent to a seed in $\ov{\mce}$, in a way that is compatible with mutation. More precisely, let $t \mapsto \ov{t}$ be an isomorphism of the labeled trees $\bbt_n$~and~$\ov{\bbt}_n$. Then $\Psi$ is a quasi-homomorphism 
if and only if 
\begin{equation}\label{qhsimdefn}
\Psi(\Sigma(t)) \sim \ov{\Sigma}(\ov{t})
\end{equation}
for all $t \in \bbt_n$, where $\Psi(\Sigma(t))= (B(t),\Psi(\mathbf{p}),\Psi(\mathbf{x}))$ is the triple obtained by evaluating $\Psi$ on $\Sigma(t)$. 
\end{defn}

As motivation for this definition, we imagine a situation where $\mce$ is well understood combinatorially, and we would like to understand  another seed pattern $\ov{\mce}$ by comparing it with $\mce$. The requirement \eqref{qhsimdefn} says that the seeds $\Psi(\Sigma(t))$ mutate ``in parallel'' with the seeds in $\ov{\mce}$, in the sense that their corresponding seeds only differ by the rescalings \eqref{rescalingxs} and \eqref{rescalingps}. 

The following Propositions \ref{basicqhpropsi} and \ref{basicqhpropsii} show two ways in which quasi-homomorphisms are well-behaved. Both of their proofs follow immediately from the observation that applying a semifield homomorphism commutes with mutation.

\begin{prop}\label{basicqhpropsi}
Let $\Psi \colon \mcf_{>0} \to \ov{\mcf}_{>0}$ be a semifield homomorphism satisfying \eqref{qhsimdefn} for \emph{some} $t \in \bbt_n$. Then $\Psi$ is a quasi-homomorphism. 
\end{prop} 

That is, rather than checking that \eqref{qhsimdefn} holds at \emph{every} $t \in \bbt_n$, it suffices to check this at a single $t \in \ov{\bbt}_n$.

\begin{prop}\label{basicqhpropsii}
Let $\Psi$  be a quasi-homomorphism from $\mce$ to $\ov{\mce}$. Let $\Sigma$ be a seed in $\mce$, and let $\Sigma^*$ be a non-normalized seed satisfying $\Sigma \sim \Sigma^*$. Then $\Psi(\Sigma) \sim \Psi(\Sigma^*) $. 
\end{prop} 

Proposition \ref{basicqhpropsii} says that quasi-homomorphism preserves $\sim$-equivalence of seeds. Thus, if $\mfs(t)$ denotes the seed orbit of $\Sigma(t)$ and ditto for $\ov{\mfs}(\ov{t})$ and $\ov{\Sigma}(\ov{t})$, then $\Psi$ maps $\mfs(t)$ inside~$\ov{\mfs}(\ov{t})$ for all $t$. A quasi-homomorphism is therefore a natural notion of homomorphism between the respective seed orbit patterns $(t, \mfs(t))$ and $(\ov{t},\ov{\mfs}(\ov{t}))$.

Now we describe a quasi-homomorphism between the pair of seed patterns in Example~\ref{GrandBandsExchangeGraphs}. 

\begin{example}\label{GrandBandsQH} Given $Y \in \bfy$, let $Y[1],Y[2],Y[3] \in \bbc^5$ denote its rows. There is a surjective map of varieties $F \colon \bfy \to \bfx$ sending $Y \overset{F}{\mapsto} Y[1] \wedge Y[2] \wedge Y[3]$. It determines a map on cluster algebras $F^* \colon \bbc[\bfx] \to \bbc[\bfy]$ sending $\Delta_{ijk} \mapsto Y_{123,ijk}$. Figure \ref{NNBandThreeFiveFig} shows the non-normalized seed pattern that arises from applying $F^*$  to Figure \ref{GrThreeFiveFig} and factoring the cluster variables inside $\bbc[\bfy]$. The seeds in Figure 
\ref{NNBandThreeFiveFig} are in the same seed orbit as the corresponding seeds in Figure \ref{BandThreeFiveFig}, and thus $F^*$ is a quasi-homomorphism from $\bbc[\bfx]$ to $\bbc[\bfy]$. 

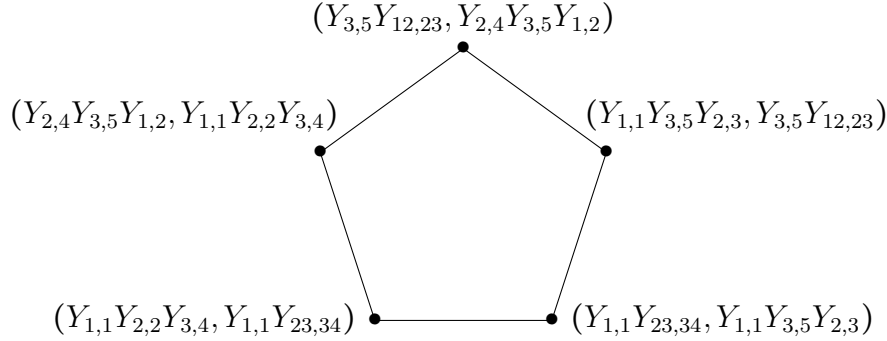
\begin{figure}[ht]
%\begin{tabular}{c}
\begin{tikzpicture}[scale = .5]
\def  \rsize{4}; 
\foreach \s in {1,...,5}
{
 \node at ({72*(-\s)+162}:\rsize) {$\bullet$};
 \draw ({72*(-\s)+162}:\rsize) -- ({72*(-\s+1)+162}:\rsize);
}

\node at (90: 1.2*\rsize) {$ (Y_{3,5} Y_{12,23},Y_{2,4}Y_{3,5} Y_{1,2})$};
\node at (17 :1.9*\rsize) {$(Y_{1,1} Y_{3,5} Y_{2,3},Y_{3,5}Y_{12,23})$};
%\node at (10 :1.6*\rsize) {$Y_{3,5}Y_{12,23})$};

\node at (-25: 1.9*\rsize) {$(Y_{1,1} Y_{23,34},Y_{1,1} Y_{3,5} Y_{2,3})$};
%\node at (-35: 1.8*\rsize) {$Y_{1,1} Y_{3,5} Y_{2,3})$};

\node at (-155: 1.9*\rsize) {$(Y_{1,1}Y_{2,2}Y_{3,4}, Y_{1,1} Y_{23,34})$};
%\node at (-135: 1.6*\rsize) {$Y_{1,1} Y_{23,34})$};

\node at (164: 2.0*\rsize) {$(Y_{2,4}Y_{3,5} Y_{1,2},Y_{1,1}Y_{2,2}Y_{3,4} )$};
%\node at (-185: 1.7*\rsize) {$Y_{1,1}Y_{2,2}Y_{3,4} )$};
\end{tikzpicture}

\[
\begin{aligned} 
Y_{2,4}Y_{3,5}Y_{1,2} \cdot Y_{1,1} Y_{3,5} Y_{2,3} &= Y_{1,1}Y_{3,5}^2Y_{2,4}(Y_{12,23}+Y_{2,2}Y_{1,3}) \\[.1cm] 
Y_{3.5}Y_{12,23} \cdot Y_{1,1} Y_{23,34}            &= Y_{1,1}Y_{3,5} (Y_{123,234} Y_{2,3}+Y_{2,2}Y_{3,3}Y_{1,3}Y_{2,4}) \\[.1cm] 
Y_{1.1} Y_{3.5} Y_{2,3} \cdot Y_{1.1}Y_{2.2}Y_{3,4} &= Y_{1.,1}^2 Y_{2,2}Y_{35} (Y_{23,34}+Y_{3,3}Y_{24}) \\[.1cm] 
Y_{1.1} Y_{23,34} \cdot Y_{24}Y_{3.5}Y_{1,2}        &= Y_{1.1}Y_{24}Y_{3,5} (Y_{2,2}Y_{1,3}  Y_{3,4}+Y_{123,234}) \\[.1cm] 
Y_{1.1}Y_{2.2}Y_{3,4} \cdot Y_{35}Y_{12,23}         &= Y_{1.1}Y_{22}Y_{3,5}(Y_{3,3}Y_{2,4} Y_{1,2}+Y_{123,234})
\end{aligned}
\]

\caption{
The non-normalized seed pattern obtained by applying $F^*$ to the seed pattern in Figure \ref{GrThreeFiveFig}. The clusters agree with the clusters in Figure \ref{BandThreeFiveFig} up to the frozen variables listed in \eqref{FrozenYs}. Cancelling the common frozen variable factors from both sides of the exchange relations yields the exchange relations in Figure \ref{BandThreeFiveFig}. It follows that the $\hat{y}$ values are the same in both figures.   
\label{NNBandThreeFiveFig}
}
\end{figure}
\end{example}

\begin{defn}\label{proportionality}
Two quasi-homomorphisms $\Psi_1,\Psi_2$ from $\mce$ to $\ov{\mce}$ are called \emph{proportional} if $\Psi_1(\Sigma) \sim \Psi_2(\Sigma)$ for all seeds $\Sigma$ in $\mce$. We say a quasi-homomorphism $\Psi$ from $\mce$ to $\ov{\mce}$ is a \emph{quasi-isomorphism} if there is a quasi-homomorphism $\Phi$ from $\ov{\mce}$ to $\mce$ such that $\Phi \circ \Psi$ is proportional to the identity map on $\mcf_{>0}$. We say that $\Psi$ and $\Phi$ are \emph{quasi-inverses} of one another.    
\end{defn}

Once we have a quasi-isomorphism between two seed patterns, we think of them as being essentially ``the same.'' Up to coefficients, the maps in both directions allows us to write the cluster variables in one seed pattern in terms of the cluster variables in the other one. 

\begin{rmk}\label{QHcategory} The set of seed patterns with quasi-homomorphisms as morphisms is a category. Proportionality is an equivalence relation on the morphisms in this category, and this equivalence relation respects composition of quasi-homomorphisms. This yields a quotient category whose objects are seed patterns and whose morphisms are proportionality classes of quasi-homomorphisms. A morphism in this quotient category is an isomorphism if and only if one (hence any) of its constituent quasi-homomorphisms is a quasi-isomorphism. 
\end{rmk}

The following lemma provides a simple method for checking a candidate map is a quasi-inverse of a given quasi-homomorphism. 

\begin{lem}\label{constructQIs}
Let $\Psi$ % \colon \mcf_{>0} \to \ov{\mcf}_{>0}$ 
be a quasi-homomorphism from $\mce$ to $\ov{\mce}$, and $\Phi \colon \ov{\mcf}_{>0} \to \mcf_{>0}$ a semifield map that preserves coefficients and for which  
$\Phi \circ \Psi(x) \asymp x$ for all cluster variables $x$ in $\mce$. Then $\Psi$ and $\varphi$ are quasi-inverse quasi-isomorphisms.  
\end{lem}

Lemma \ref{constructQIs} follows from the more general Proposition \ref{checkqhonnerves} below. In fact, it will suffice to merely check that $\varphi \circ \Psi(x) \asymp x$ for all $x$ lying on a \emph{nerve} (cf.~Definition \ref{NerveDefn}).  

\begin{example}\label{GrandBandsQI} 
Using Lemma \ref{constructQIs} we describe a quasi-inverse $G^*$ for the quasi-~homomorphism $F^*$ from Example \ref{GrandBandsQH}. Let $G \colon \bfx \to \bfy$ be the morphism sending $X \in \bfx$ to the the band matrix
$$G(X) =  \begin{pmatrix}
\Delta_{145}(X) & \Delta_{245}(X) & \Delta_{345}(X) & 0 & 0 \\
0& \Delta_{125}(X) & \Delta_{135}(X) & \Delta_{145}(X) &  0 \\
0& 0 & \Delta_{123}(X) & \Delta_{124}(X) &  \Delta_{125}(X) \\
\end{pmatrix} \in \bfy$$
all of whose entries are Pl\"ucker coordinates of $X$. The coordinate ring map $\bbc[\bfy] \to \bbc[\bfx]$ sends $Y_{i,j}$ to the Pl\"ucker coordinate in the $(i,j)$ entry of $G(X)$, e.g. $G^*(Y_{1,2}) = \Delta_{245}$. 

The matrix $G(X)$ has an interesting property: all of its minors are monomials in the Pl\"ucker coordinates of $X$. In particular, its maximal minors agree with those of $X$, up to a multiplicative factor:
\begin{equation}\label{FGcomp}
\Delta_{ijk}(G(X)) = \Delta_{145}(X) \Delta_{125}(X) \Delta_{ijk}(X).
\end{equation} 

Thus, $G^* \circ F^*(\Delta_{ijk}) = \Delta_{145}\Delta_{125}\Delta_{ijk} \asymp \Delta_{ijk}$ for each cluster variable $\Delta_{ijk}$. Since $G^*$ preserves coefficients (the only nontrivial check is $G^*(Y_{123,234}) = \Delta_{125}\Delta_{145}\Delta_{234}$), from Lemma \ref{constructQIs} it follows that $G^*$ is a quasi-inverse of $F^*$. 
\end{example}

\begin{rmk}\label{largersubring}
A quasi-homomorphism is defined as a map on ambient semifields since these maps transparently preserve the mutation rules \eqref{brecurrence}--\eqref{nnxrecurrence}. This should be suitable for most purposes, since one is mostly interested in evaluating a quasi-homomorphism on cluster variables or coefficients. However, the cluster algebra~$\mca$ is the more familiar algebraic object associated to a seed pattern. If one wants to think of a quasi-homomorphism $\Psi$ as an algebra map of cluster algebras $\mca \to \ov{\mca}$, one will sometimes need to first localize at the frozen variables in $\ov{\mca}$. 
\end{rmk}

We close this section by explaining the differences between quasi-homomorphisms and preexisting notions of homomorphisms between cluster algebras. Specifically, we consider the notion of a \emph{rooted cluster morphism} in the category of cluster algebras described by Assem, Dupont, and Shiffler \cite{ADS}, and also of a \emph{coefficient specialization} defined by Fomin and Zelevinsky \cite{CAIV} and studied by Reading \cite{Reading,ReadingSurfaces}. The key difference between these notions and quasi-homomorphisms is that a quasi-homomorphism allows for cluster variables to be \emph{rescaled} by an element of $\ov{\bfp}$. This extra flexibility provides more freedom in constructing new cluster algebras from old ones (cf.~Section \ref{NormalizedQHs}) or in finding nice self-maps of cluster algebras giving rise to elements of the cluster modular group (cf.~Section~\ref{QASecn}).

In a little more detail, a coefficient specialization is a map whose underlying map on coefficients can be any group homomorphism $\bfp \to \ov{\bfp}$, but that must send each cluster variable to a cluster variable. Thus, each coefficient specialization is a quasi-homomorphism, but a very special one since cluster variables are not allowed to be rescaled by elements of $\ov{\bfp}$. Rooted cluster morphisms require choosing a pair of initial seeds in $\mce$ and $\ov{\mce}$ (this is the sense in which the morphism is rooted). Between this pair of seeds, a morphism is an algebra map that sends each cluster variable to either a cluster variable or an integer, and sends each frozen variable to either a frozen variable, a cluster variable, or an integer. Hence, while quasi-homomorphisms are more flexible in allowing for cluster variables to be rescaled and for frozen variables to be sent to \emph{monomials} in the frozen variables, they are also less flexible as they do not allow for unfreezing frozen variables or specializing variables to integers. It probably would not be hard to combine these two notions.    

One more technicality: to streamline the discussion, we have formulated Definition \ref{qhdefn} so that quasi-homomorphisms preserve exchange matrices, whereas rooted cluster morphisms allow for $B \mapsto -B$. To make Definition \ref{qhdefn} more consonant with these preexisting notions, one could modify \eqref{qhsimdefn} to say that either $\Psi(\Sigma(t)) \sim \Sigma(t)$ or $\Psi(\Sigma(t)) \sim \Sigma(t)^{\text{opp}}$ (see the definition of \emph{opposite seed} in Section \ref{NervesSecn} below) without any significant changes.

\section{Normalized seed patterns}\label{NormalizedQHs}
In this section, we recall the definition of normalized seed patterns and apply the results of Section \ref{SeedOrbitsSecn} in the case that $\mce$ and $\ov{\mce}$ are normalized. 

\begin{defn}[Normalized seed pattern]\label{normalizedpatterndefn}
A seed pattern $\mce$ as in Definition \ref{nnseedpatterndefn} is called \emph{normalized} if the coefficient group $\bfp$ is a semifield, and 
each coefficient tuple $\mathbf{p}(t)$ satisfies 
\begin{equation}\label{normalized}
p^+_j(t) \oplus p^-_j(t) = 1 \text{ for all $j$},
\end{equation}
where $\oplus$ is the addition in $\bfp$. 
\end{defn}

The advantage of this normalization condition is that it makes the mutation rule \eqref{nnyrecurrence}, and therefore mutation of normalized seeds, unambiguous.  Indeed, \eqref{nnyrecurrence} specifies the ratio $y_j(t') = \frac{p^+_j(t')}{p^-_j(t')}$ in terms of $B(t)$ and $\mathbf{p}(t)$, and there is a unique choice of $p^\pm_j(t') \in \bfp$ with this ratio and satisfying the normalization condition, namely the pair   
\begin{equation}\label{ysandps}
p^+_j(t') = \frac{y_j(t')}{1 \oplus y_j(t')} \text{ and } p^-_j(t') = \frac{1}{1 \oplus y_j(t')}. 
\end{equation}
Furthermore, mutating twice in a given direction is the identity. 

At the same time, the disadvantage is that computing a cluster algebra now involves three operations, the two operations present in $\mcf_{>0}$ along with $\oplus$ in $\bfp$. The definition of quasi-homomorphism prioritizes these first two operations. Proposition \ref{qhnormalizedprop} says that in the case of a quasi-homomorphism between normalized seed patterns, there is a ``separation of additions'' phenomenon, separating the addition in $\ov{\mcf_{>0}}$ from the one in $\ov{\bfp}$.

\medskip
Before stating Proposition \ref{qhnormalizedprop}, we say a little more about normalized seed patterns and~\emph{$Y$-patterns}. In a normalized seed pattern, the tuple of ratios $(y_1(t),\dots,y_n(t))$ determines the coefficient tuple by \eqref{ysandps}. Accordingly, for normalized seed patterns one keeps track of $y_j(t)$ rather than $p^\pm_j(t)$. 
Rewriting \eqref{nnyrecurrenceatk} and \eqref{nnyrecurrence} in terms of $y_j(t)$  determines a \emph{$Y$-pattern recurrence in the semifield $\bfp$}:
\begin{equation}\label{yrecurrence}
y_{j}(t') = 
\begin{cases}
y_j(t)^{-1} & \text{ if } j = k  \\
y_j(t) y_k(t)^{[b_{kj}(t)]_+}(y_k(t)\oplus 1)^{-b_{kj}(t)} & \text{ if } j \neq k. 
\end{cases}
\end{equation}

A collection of quantities $(B(t),\mathbf{y}(t))_{t \in \bbt_n}$ satisfying \eqref{brecurrence} and \eqref{yrecurrence}, with the $\mathbf{y}(t)$ lying in some semifield $\mcs$, is called a \emph{Y-pattern in the semfield $\mcs$}. Notice that the concept of semifield is now playing two different roles, either as the ambient semifield $\mcf_{>0}$ in which the exchange relation calculations take place, or as the coefficient semifield $\bfp$ used to remove the ambiguity in mutation of seeds. The surprising connection between these two roles is Lemma \ref{nnyhatsmutate}, which we now recognize as saying that the $(B(t),\mathbf{\hat{y}}(t))$ form a~$Y$-pattern in the ambient semfield $\mcf_{>0}$.

The most important example of a coefficient semifield arising in applications is the \emph{tropical semifield}. 
\begin{defn}
A \emph{tropical semifield} is a free abelian multiplicative group in some generators $u_1,\dots,u_m$, with auxiliary addition $\oplus$ given by 
$$\prod_{j} u_j^{a_j} \oplus \prod_{j} u_j^{b_j} = \prod_{j} u_j^{\min (a_j,b_j)}.$$
\end{defn}

A normalized seed pattern over a tropical semifield is said to be  \emph{of geometric type}. When this is the case, denoting the frozen variables by $x_{n+1},\dots,x_{n+m}$, the data of $(B(t),\mathbf{y}(t))$ is entirely described by an $(n+m) \times n$ matrix $\tilde{B}(t) = (b_{ij}(t))$, called the \emph{extended exchange matrix}. Its top $n \times n$ submatrix is $B(t)$ and is called the \emph{principal part}. Its bottom $m$ rows are called \emph{coefficient rows}. They are specified from the equality $\displaystyle y_j(t) = \prod_{i=1}^m x_{n+i}^{b_{n+i,j}(t)}$.  The mutation rule \eqref{yrecurrence} translates into the rule \eqref{brecurrence} on $\tilde{B}$.

\medskip
The seed pattern in Figure \ref{GrThreeFiveFig} is of geometric type over the tropical semifield in the frozen variables in \eqref{FrozenXs}. The same holds for the seed pattern in Figure \ref{BandThreeFiveFig} over the frozen variables in \eqref{FrozenYs}. On the other hand, the seed pattern in Figure \ref{NNBandThreeFiveFig} is not normalized, e.g. the first exchange relation there satisfies $p^+_2 \oplus p^-_2 = Y_{1,1}Y_{2,4}Y_{3,5}$. 
\medskip

Now we state Proposition \ref{qhnormalizedprop} describing quasi-homomorphisms between normalized seed patterns $\mce$ and $\ov{\mce}$. It arose during the process of writing a forthcoming book on cluster algebras \cite{CABook}, in proving one direction of the finite type classification (namely, that a cluster algebra with a quiver whose principal part is an orientation of a Dynkin quiver necessarily has only finitely many seeds).  We will state it as a recipe for \emph{constructing} a normalized seed pattern from a given one, since we envision this being useful in applications.

\begin{prop}\label{qhnormalizedprop}
Let $\mce$ be a non-normalized seed pattern, with the usual notation.  Let $(x_i) = (x_i(t_0))$ be a fixed initial cluster in $\mce$. Let $\ov{\bfp}$ be a semifield, and $\ov{\mcf}_{>0}$ the semifield of subtraction-free rational expressions in algebraically independent elements $\ov{x}_1,\dots,\ov{x}_n$ with coefficients in $\ov{\bfp}$.

Let $\Psi \colon \mcf_{>0} \to \ov{\mcf}_{>0}$ be a semifield map satisfying $\Psi(x_i) \asymp \ov{x}_i$, and let $c \colon \mcf_{>0} \to \ov{\bfp}$ be the composition of semifield maps $\mcf_{>0} \xrightarrow{\Psi} \ov{\mcf}_{>0} \xrightarrow{\ov{x}_i \mapsto 1} \ov{\bfp}$ where the second map in this composition specializes all $\ov{x}_i$ to  $1$ and is the identity on $\ov{\bfp}$.  

Then there is a normalized seed pattern $\ov{\mce}$ in $\ov{\mcf}_{>0}$ with seeds $(\ov{B}(\ov{t}),\ov{\mathbf{p}}(\ov{t}), \ov{\mathbf{x}}(\ov{t}))$ satisfying 
\begin{align}
\ov{B}(\ov{t}) &= B(t) \label{seedhomtopatternhomeqsi} \\
\ov{x}_i(\ov{t}) &= \frac{\Psi(x_i(t)) }{c(x_i(t))}  \label{seedhomtopatternhomeqsiv}\\
\hat{\ov{y}}_i(\ov{t}) & = \Psi(\hat{y}_i(t)) \label{seedhomtopatternhomeqsiii}\\
\ov{y}_i(\ov{t}) &= c(\hat{y}_i(t)) \label{seedhomtopatternhomeqsii}.
\end{align}
Clearly, $\Psi$ is a quasi-homomorphism from $\mce$ to $\ov{\mce}$.
\end{prop}

\begin{proof}
Formulas \eqref{seedhomtopatternhomeqsiv} through \eqref{seedhomtopatternhomeqsii} follow by applying \cite[Proposition $3.4$]{CATSII}) to $\Psi(\mce)$, renormalizing by the scalars $c(x_i)$, and massaging the formulas given there. Alternatively, it is also straightforward to check the right hand sides of \eqref{seedhomtopatternhomeqsiv} through \eqref{seedhomtopatternhomeqsii} satisfy the required recurrences directly (this is carried out in \cite{CABook}).   
\end{proof}

\begin{example} Beginning with the seed pattern in Figure \ref{GrThreeFiveFig}, one can construct the normalized seed pattern in Figure \ref{BandThreeFiveFig} by first applying the semifield map $F^*$ -- obtaining the non-normalized seed pattern in Figure \ref{NNBandThreeFiveFig}-- and then normalizing by a semifield map~$c \colon \mcf_{>0} \to \ov{\bfp}$. This map $c$ agrees with $F^*$ on frozen variables and sends a cluster variable $x$ to the frozen variable monomial dividing $F^*(x)$, e.g. $c(\Delta_{235}) = Y_{3,5}$ and $c(\Delta_{245}) = Y_{2,4}Y_{3,5}$.  
\end{example}

When both $\mce$ and $\ov{\mce}$ are of geometric type, constructing a quasi-homomorphism that sends one seed into (the seed orbit of) another seed is a matter of linear algebra: 
\begin{cor}\label{geomtypeseedhom}
Let $\Sigma = (\tilde{B},\{x_1,\dots,x_n\})$ and $\ov{\Sigma} = (\ov{\tilde{B}},\{\ov{x}_1,\dots,\ov{x}_n \})$ be seeds of geometric type, with frozen variables  
$x_{n+1},\dots,x_{n+m}$  and $\ov{x}_{n+1},\dots,\ov{x}_{n+\ov{m}}$ respectively. Let $\mce$ and $\ov{\mce}$ be the respective seed patterns. 

Let $\Psi$ be a quasi-homomorphism from $\mce$ to $\ov{\mce}$ such that $\Psi(\Sigma) \sim \ov{\Sigma}$. It determines a monomial map from the $x_i$ to the $\ov{x}_i$. Let $M_\Psi$ denote the matrix of exponents of this monomial map, thus $M_\Psi$ is an $(n+\ov{m}) \times (n+m)$ matrix satisfying $\displaystyle \Psi(x_k) = \prod_{i=1}^{n+{\ov{m}}}\ov{x}_i^{(M_\Psi)_{ik}}$. Then the extended exchange matrices $\tilde{B},\ov{\tilde{B}}$ are related by 
\begin{equation}\label{newbtilde}
\ov{\tilde{B}}= M_\Psi \tilde{B}.
\end{equation}

In particular, such a a quasi-homomorphism $\Psi$ exists if and only if the principal parts of $\tilde{B},\ov{\tilde{B}}$ agree,
and the (integer) row span of $\tilde{B}$ contains the (integer) row span of $\ov{\tilde{B}}$.
\end{cor}

\begin{proof}
Indeed, the $(i,j)$ entry of the left hand side of \eqref{newbtilde}
encodes the exponent of $\ov{x}_i$ in $\hat{\ov{y}}_j$, while the $(i,j)$ entry of the 
right hand side encodes the exponent of $\ov{x}_i$ in $\Psi(\hat{y}_j)$. So \eqref{newbtilde} now follows from \eqref{seedhomtopatternhomeqsiii}.

The final statement follows by studying \eqref{newbtilde}: the ``interesting'' rows of $M_\Psi$ are its bottom $\ov{m}$ rows. 
Each of these rows determines a particular linear combination of the rows of $\tilde{B}$, and these linear combinations can be prescribed arbitrarily by prescribing the exponent of $\ov{x}_i$ in $\Psi(x_j)$ for $1 \leq i \leq \ov{m}, 1 \leq j \leq n+m$ using Lemma \ref{universalsemifieldlemma}. 
\end{proof}

\begin{rmk}[Exchange graphs and separation of additions]\label{separationofadditionsrmk} 
The formulas  \eqref{seedhomtopatternhomeqsi},\eqref{seedhomtopatternhomeqsiv} and \eqref{seedhomtopatternhomeqsii} show that if there is a quasi-homomorphism from $\mce$ to $\ov{\mce}$, then the exchange graph of $\mce$ covers that of $\ov{\mce}$. In particular, by Corollary \ref{geomtypeseedhom}, if the rows of $\tilde{B}$ span $\bbz^n$, then the exchange graph for the corresponding cluster algebra $\mca(\tilde{B})$ covers the exchange graph of every other cluster algebra $\ov{\mca}$ with the same underlying exchange matrix. 

This is a natural generalization of the \emph{separation of additions formula} \cite[Theorem $3.7$]{CAIV} from the case of a quiver with principal coefficients to \emph{any} $\tilde{B}$-matrix whose rows span~$\bbz^n$. Namely, let~$\Sigma_0 = (B_0,\mathbf{y},\mathbf{x})$ and~$\ov{\Sigma}_0 =  (B_0,\ov{\mathbf{y}},\ov{\mathbf{x}})$ be a pair of normalized seeds with the same exchange matrix, and suppose~$\Sigma_0$ has principal coefficients, i.e.~$y_i = x_{n+i}$. There is a natural choice of maps~$\Psi$ mapping~$\Sigma_0$ to~$\ov{\Sigma_0}$ as in Proposition \ref{qhnormalizedprop}, defined by~$\Psi(x_i) =  \ov{x}_i$ and~$\Psi(x_{n+i}) = \ov{y}_i$. For this choice of~$\Psi$, formula~\eqref{seedhomtopatternhomeqsiv} becomes separation of additions: the numerator of~\cite[Theorem $(3.7)$]{CAIV} (evaluating the~``$X$ polynomial'' in~$\ov{\mcf}$) is applying the semifield homomorphism~$\Psi$, while the denominator (specializing the cluster variables to~$1$ and evaluating the~$X$ polynomial in~$\ov{\bfp}$) is applying the semifield map~$c$.
\end{rmk}

\begin{rmk}[Proportionality and gradings]\label{propandgradings} Let $\mce$ be a seed pattern of geometric type. We recall briefly the concept of a \emph{$\bbz^r$-grading on $\mce$} cf.~\cite{Grabowski, GrabowskiLaunois}. Choosing an initial seed $(\tilde{B},\{x_i\})$
in $\mce$, such a choice of grading is determined by a $r \times (n+m)$ grading matrix $G$ satisfying $G \tilde{B} = 0$. The $i^{\text{th}}$ column of $G$ determines the grading of $x_i$ as a vector in $\bbz^r$, for $1 \leq i \leq n+m$. The condition $G \tilde{B} = 0$ guarantees that every exchange relation \eqref{nnxrecurrence} is homogeneous with respect to this $\bbz^r$-grading; this in turn defines the multi-grading of each adjacent cluster variable and thereby each adjacent grading matrix. It can be seen that these adjacent grading matrices again satisfy the left kernel condition, so that the grading propagates to a $\bbz^r$-grading on the entire cluster algebra in which the cluster variables and coefficients are homogeneous. 

Now we suppose we are given two seeds $\mce$ and $\ov{\mce}$ of geometric type with notation as in Corollary \ref{geomtypeseedhom}. Let $\Psi_1$ and $\Psi_2$ be a pair of proportional quasi-homomorphisms of $\mce$ and~$\ov{\mce}$. We obtain as in \eqref{newbtilde} matrices $M_{\Psi_1}$ and $M_{\Psi_2}$ such that $M_{\Psi_1} \tilde{B} = M_{\Psi_2} \tilde{B} = \ov{\tilde{B}}$, which implies that $M_{\Psi_1}-M_{\Psi_2}$ defines a $\bbz^{\ov{m}}$-grading $G$ on $\mce$ (the first $n$ rows of $M_{\Psi_1}-M_{\Psi_2}$ define the trivial grading). Conversely, fixing a quasi-homomorphism $\Psi_1$ with matrix $M_{\Psi_1}$, any choice of $\bbz^{\ov{m}}$-grading matrix $G$ on $\tilde{B}$ provides a quasi-homomorphism $\Psi_2$, proportional to~$\Psi_1$, whose matrix is $M_{\Psi_2} = M_{\Psi_1} + G$. 
\end{rmk}

\begin{rmk}\label{QRowSpans}
For simplicity, we stated Corollary \ref{geomtypeseedhom} in terms of $\bbz$ row spans, but a similar statement holds for  $\bbq$ row spans. To do this, one enlarges the tropical semifield $\ov{\bfp}$ to the \emph{Puiseux tropical semifield} consisting of Puiseux monomials with rational exponents in the frozen variables. This is unpleasant from the perspective of cluster algebras as coordinate rings, but is perfectly fine if one is only interested in writing algebraic formulas for cluster variables, etc. 

%In a similar vein, it would appear that the generalization of the separation of additions formula stated in the preceding remark only holds when the second cluster algebra $\ov{\mce}$ is of geometric type. In fact, there is a workaround: Corollary \ref{geomtypeseedhom} holds even when $\ov{\bfp}$ is a semifield in \emph{infinitely many generators} $\ov{x}_{n+1},\ov{x}_{n+2},\dots$, and this immediately implies that the Corollary holds for an arbitrary choice of semifield $\ov{\bfp}$ (by choosing a surjection of multiplicative groups from the tropical semifield onto the given one). 

This is foreshadowed in the work of Sherman and Zelevinsky \cite[Section 6]{Sherman}, which discusses the coefficient-free rank 2 cluster algebra $\mca(b,c)$ with exchange matrix 
$\begin{pmatrix}
0 & a \\
-b & 0 
\end{pmatrix}
$. The authors write the cluster variables in any cluster algebra with this $B$ matrix in terms of the cluster variables for $\mca(b,c)$. Their formulas involve Puiseux monomials in the frozen variables. 
\end{rmk}

\section{Nerves}\label{NervesSecn}
By Proposition \ref{basicqhpropsi}, to check that a given semifield map $\Psi$ is a quasi-homomorphism from $\mce$ to $\ov{\mce}$, it suffices to check that $\Psi(\Sigma(t)) \sim \ov{\Sigma}(\ov{t})$ for \emph{some} pair of seeds $\Sigma(t)$ in $\mce$ and $\ov{\Sigma}(\ov{t})$. By Proposition \ref{equivalentseedscondition}, this means checking that $B(t) = B(\ov{t})$ and  $\hat{\mathbf{y}}(t) = \hat{\mathbf{y}}(\ov{t})$, and furthermore $\Psi(x_j(t)) \asymp \ov{x}_j(\ov{t})$ holds for all $j$. We envision applications where checking the proportionality condition on cluster variables is easy and can be done in \emph{many} seeds $t$, but checking the equality of exchange matrices or $\hat{\mathbf{y}}$'s is inconvenient. The goal of this section is to give a criterion that guarantees $\Psi$ is a quasi-homomorphism by only checking these proportionality conditions. The relevant concept is that of a \emph{nerve} for a seed pattern. 

\begin{defn}\label{NerveDefn}
Let $\mce$ be a seed pattern.  A \emph{nerve}~$\mcn$ for~$\bbt_n$, is a connected subgraph of~$\bbt_n$ such that every edge label~$k \in [1,n]$ arises at least once in $\mcn$.   
\end{defn}

The basic example of a nerve is the star neighborhood of a vertex. We believe that there are many theorems of the form, ``if a property holds on a nerve, then it holds on the entire seed pattern.'' We give an example of such a theorem in the appendix, generalizing the ``Starfish Lemma'' \cite[Proposition 3.6]{tensors} from a star neighborhood to a nerve.

Before stating the result of this section, we need to address the (mostly unimportant) difference between a seed and its opposite seed. We say a seed is \emph{indecomposable} if the underlying graph described by its exchange matrix (the vertex set is $[1,n]$ and vertices $i,j$ are joined by an edge if $b_{i,j} \neq 0$) is connected. For a seed $\Sigma = (B,\mathbf{p},\mathbf{x})$, the \emph{opposite seed}~$\Sigma^{\text{opp}} = (B^{\text{opp}},\mathbf{p}^{\text{opp}}, \mathbf{x}^{\text{opp}})$ is the seed defined by $B^{\text{opp}} = -B$, $(p^{\text{opp}})^\pm_j = p^\mp_j$, and $x^{\text{opp}}_i = x_i$. It satisfies~$\hat{y}^{\text{opp}}_j = \frac{1}{\hat{y}_j}$. The operations of restricting to an indecomposable component and replacing a seed by its opposite seed both commute with mutation.

\begin{prop}\label{checkqhonnerves} Let $\mce$ and $\ov{\mce}$ be non-normalized seed patterns, with respective ambient semifields $\mcf_{>0}$ and $\ov{\mcf}_{>0}$. Suppose the seeds in $\ov{\mce}$ are indecomposable.  Let $\Psi \colon \mcf_{>0} \to \ov{\mcf}_{>0}$ be a semifield homomorphism that preserves coefficients and satisfies $\Psi(x_j(t)) \asymp \ov{x}_j(\ov{t})$ for every vertex $t$ and label $j$ such that $t \xrightarrow{j} t'$ is in  $\mcn$. 
Then $\Psi$ is a quasi-homomorphism from $\mce $ to $\ov{\mce}$ or from $\mce$ to $\ov{\mce}^{\text{opp}}$. 
\end{prop}

In particular applying the proposition when $\mcn$ is the star neighborhood of a vertex $t$, to check that $\Psi$ is a quasi-homomorphism, it suffices to check that $\Psi(x_j(t)) \asymp \ov{x}_j(\ov{t})$ for all $j \in [1,n]$, as well as checking $\Psi(x_j(t')) \asymp \ov{x}_j(\ov{t'})$ for each adjacent edge $t \xrightarrow{j} t'$. Lemma \ref{constructQIs} now follows. 

\begin{proof}
Choose a vertex $t \in \mcn$. By hypothesis for all $j$, $\Psi(x_j(t)) = c_j(t) \ov{x}_j(\ov{t})$ for some $c_j(t) \in \ov{\bfp}$, so we are left checking that $B(t) = B(\ov{t})$ and $\hat{\mathbf{y}}(t) = \hat{\mathbf{y}}(\ov{t})$. Suppose $t \xrightarrow{k} t'$ is an edge in $\mcn$, then there is a scalar $c_k(t')$ such that $\Psi(x_k(t' )) = c_k(t') \ov{x}_k(\ov{t'})$.  

The exchange relation defining $\ov{x}_k(\ov{t'})$ in $\ov{\mce}$ is
\begin{equation}\label{barredexchange}
\ov{x}_{k}(\ov{t})\ov{x}_{k}(\ov{t'}) =  \ov{p}^+_k(\ov{t}) \prod \ov{x}_{j}(\ov{t})^{[\ov{b}_{jk}(\ov{t})]_+} 
 + \ov{p}^-_k(\ov{t})\prod \ov{x}_{k}(\ov{t})^{[-\ov{b}_{jk}(\ov{t})]_+}.
\end{equation}
 
On the other hand, applying $\Psi$ to the relation defining $x_k(t')$ in $\mce$ and rearranging yields 
\begin{equation}\label{applyPsitoexchange}
\ov{x}_k(\ov{t})\ov{x}_{k}(\ov{t'}) = \frac{1}{c_k(t')c_k(t)} ( \Psi(  p^+_k(t)) \prod \Psi(x_j(t))^{[b_{jk}(t)]_+}
+\Psi(p^-_k(t)) \prod (\Psi(x_j(t)))^{[-b_{jk}(t)]_+} ).
\end{equation}

Abbreviating the two terms on the right hand side of \eqref{barredexchange} as $X+Y$, and the two terms in \eqref{applyPsitoexchange} as $Z+W$, 
we see by algebraic independence in the seed at $t$ that either $X = Z, Y = W$, or $X = W, Y = Z$. Refer to these as Case $1$ or Case $2$ respectively. 
By inspection, we see that $\hat{\ov{y}}_k(\ov{t})$ is the ratio $\frac{X}{Y}$, while  $\Psi(\hat{y}_k(t))$ is $\frac{Z}{W} $. Thus in Case $1$ we deduce that $\Psi(\hat{y}_k(t)) = \hat{\ov{y}}_k(\ov{t})$ and the matrices $B(t)$ and $B(\ov{t})$ have the same $k^{\text{th}}$ column. In Case $2$ we deduce the same thing once we replace $\ov{\mce}$ by $\ov{\mce}^{\text{opp}}$. 

Now apply Lemma \ref{nerveYsyslem}.
\end{proof}

\begin{lem}\label{nerveYsyslem}
Let $\mcy = \{\mathbf{y}(t),B(t)\}$ and $\ov{\mcy} = \{\ov{\mathbf{y}}(t),\ov{B}(t)\}$ be two $Y$-patterns whose matrices $B(t)$ are indecomposable. Let $\mcn$ be a nerve for $\bbt_n$. Suppose for every vertex $t \in \mcn$ and label $k$ such that the edge $t \xrightarrow{k} t'$ is in $\mcn$, one of the following holds 
\begin{align}
y_k(t) = \ov{y}_k(t)  &\text{ and } b_{jk}(t) = \ov{b}_{jk}(t) \text{ for all $j \in [1,n]$, or} \label{epartsagree}\\
y^{\text{opp}}_k(t) = \ov{y}_k(t)  &\text{ and } b^{\text{opp}}_{jk}(t) = \ov{b}_{jk}(t) \text{ for all $j \in [1,n]$}, \label{epartsagreeii}
\end{align}
then $\mcy = \ov{\mcy}$ or $\mcy^{\text{opp}} = \ov{\mcy}$ accordingly. 
\end{lem}

Roughly, there are two issues here: first the question of whether $Y$-patterns can be checked on a nerve (they can), and second whether we are dealing with $\mcy$ or $\mcy^{\text{opp}}$ (this relies on indecomposability). 

\begin{proof}
In any $Y$-pattern, for a given $(k,t)$ pair (not necessarily in $\mcn$), we will refer to $y_k(t)$ and the $k^{\text{th}}$ column of $B(t)$  as the \emph{$k$-part of the seed at $t$}. The equations in \eqref{epartsagree} say that $\mcy,\ov{\mcy}$ have either the same $k$-parts, or opposite $k$-parts, for any edge  $t \xrightarrow{k} t' \in \mcn$. 

Pick a vertex $t_0 \in \mcn$, and an edge $k \in \mcn$ incident to $t_0$. If necessary, replace $\mcy$ by $\mcy^{\text{opp}}$ so that 
the given $Y$-patterns have the same $k$-part at $t_0$. We seek to prove $\mcy,\ov{\mcy}$ have the same $j$-part at $t_0$, for all $j \in [n]$. 
 
Let $t_0 \xrightarrow{k} t_1 \in \mcn$ be an edge in the nerve incident to $t_0$. 
The mutation rules \eqref{brecurrence}, \eqref{yrecurrence} are involutive %commute with replacing $\mcy$ with $\mcy^{\text{opp}}$, 
and have the property that for any $j$, the $j$-part of the seed at $t_1$ depends only on the $j$-part and $k$-part of the seed at $t_0$. Since the given $Y$-patterns agree at $k$, we see that their $j$-parts agree at $t_1$ if and only if they agree at $t_0$. Repeatedly apply this observation, mutating in all possible directions in the nerve, while preserving the fact that the $j$-parts at $t \in \mcn$ coincide if and only if they coincide at $t_0$. Since the nerve is connected and every edge label shows up at least once in $\mcn$, we conclude that for all $j$, the $j$-parts at $\mcy$ and $\ov{\mcy}$ are either the same or opposite. The connectedness hypothesis assures they are all in fact the same.  
\end{proof}

\section{Quasi-automorphisms and the cluster modular group}\label{QASecn}
A \emph{quasi-automorphism} is a quasi-isomorphism from a given seed pattern $\mce$ to itself, cf.~Definition~\ref{proportionality}. One can think of a quasi-automorphism as a choice of a map describing an automorphism of the pattern of seed orbits associated to $\mce$. We will use quasi-automorphisms to define a variant of a group of automorphisms of $\mce$, generalizing the group of cluster automorphisms defined for seed patterns with trivial coefficients in \cite{ASS} while retaining many of the properties of cluster automorphisms (e.g. Proposition \ref{basicqhpropsi}, Corollary \ref{geomtypeseedhom} and Proposition \ref{checkqhonnerves}).  

The following example illustrates that the notion of quasi-automorphism is more general than the ``naive'' notion of a semifield automorphism preserving the seed orbit pattern. 
\begin{example}\label{QAsarenotautos} A quasi-automorphism does not have to be an automorphism of semifields. Consider the composition $G^* \circ F^*$ from Example \ref{GrandBandsQI}, which is a quasi-automorphism of~$\bbc[\bfx]$ proportional to the identity map. It rescales each Pl\"ucker variable by a product of frozens: $G^* \circ F^* (\Delta_S) = \Delta_{145}\Delta_{125}\Delta_S$. The ambient semifield of $\bbc[\bfx]$ has a grading for which every Pl\"ucker variable is degree one, and every homogeneous element in the image of $G^* \circ F^*$ has degree a multiple of $3$. Thus $G^* \circ F^*$ cannot be surjective. 
\end{example}

\begin{defn}\label{QAutDefn} The \emph{quasi-automorphism group} $\QAut_0(\mce)$ is the set of proportionality classes of quasi-automorphisms of $\mce$. This is the automorphism group of $\mce$ in the quotient category discussed in Remark \ref{QHcategory}. 
\end{defn}

\begin{rmk}\label{CMGandCoeffs} Let us call a quasi-automorphism trivial if it is proportional to the identity map. The set of trivial quasi-automorphisms is a monoid (but not usually a group) under composition; the composition $G^* \circ F^*$ from Example \ref{GrandBandsQI} bears witness to this. One way to construct quasi-automorphisms proportional to a given~$\Psi$ is to compose $\epsilon_1 \circ \Psi \circ \epsilon_2$ with $\epsilon_1$ and $\epsilon_2$ trivial. It is tempting to try and define $\QAut_0(\mce)$ purely in terms of thse trivial quasi-automorphisms, without mentioning proportionality. However the relation~$\equiv$ defined by $\Psi_1 \equiv \Psi_2$ if $\Psi_2 = \epsilon_1 \circ \Psi \circ \epsilon_2$ is neither symmetric nor transitive, so one cannot form a quotient category using this relation. 
\end{rmk}

We write~$\QAut_0(\mce) = \QAut_0(\tilde{B})$ when $\mce$ is of geometric type and specified by an initial matrix $\tilde{B}$. By Remark \ref{propandgradings}, two quasi-automorphisms are proportional to each other if and only if their ratio defines a $\bbz^m$-grading on $\mce$ (taking exponents of elements of $\bfp$ to obtain elements of $\bbz^m$). Fixing a particular quasi-automorphism $\Psi$, the number of degrees of freedom in specifying another quasi-automorphism proportional to $\Psi$ is therefore the corank of~$\tilde{B}$.

\begin{lem}\label{symmetryofspans} Let $\mce$ be a seed pattern of geometric type and $\Psi$ a quasi-homomorphism from $\mce$ to itself. Then $\Psi$ is a quasi-automorphism.
\end{lem}

Thus when $\mce$ is of geometric type, every quasi-homomorphism $\Psi$ from $\mce$ to itself determines an element of $\QAut_0(\mce)$, i.e. any such $\Psi$ has a quasi-inverse.

\begin{proof}  By \cite[Lemma 3.2]{CAIII}, if two $\tilde{B}$-matrices $\tilde{B}(t_0)$ and $\tilde{B}(\ov{t}_0)$ are in the same mutation class, they are related by a pair of unimodular integer matrices: $ \tilde{B}(\ov{t}_0) = M \tilde{B}(t_0) N$, for~$ M \in \GL_{m+n}(\bbz),$ and~$N \in \GL_n(\bbz)$.

By Corollary \ref{geomtypeseedhom}, for a pair of vertices $t_0,\ov{t_0} \in \bbt_n$, there is a quasi-homomorphism $\Psi$ sending the seed orbit at $t_0$ to the seed orbit at $\ov{t}_0$ if and only if the principal parts of $\tilde{B}(t_0)$ and $\tilde{B}(\ov{t}_0)$ agree, and the row span of $\tilde{B}(t_0)$ contains the row span of $\tilde{B}(\ov{t}_0)$. By the unimodularity of mutation, this criterion is preserved under swapping the roles of $t_0$ and $\ov{t}_0$ -- if the row span of $\tilde{B}(t_0)$ contains the row span of $\tilde{B}(\ov{t}_0)$ then in fact the two row spans are equal submodules of $\bbz^n$. 
\end{proof}

\begin{rmk}\label{twistremark} Marsh and Scott \cite{MarshScott} described a version of the twist for the Grassmannian cluster algebras. One can show that it is a quasi-automorphism using \cite[Corollary 8.6]{MarshScott}. 
\end{rmk}

We will now recall the definitions of some preexisting groups of automorphisms associated a seed pattern $\mce$. Namely: 
\begin{itemize}
\item the cluster modular group $\CMG(\mce)$ of Fock and Goncharov \cite{FG}, and
\item the group $\Aut(\mce)$ of automorphisms in the category of (rooted) cluster algebras defined by Assem, Dupont and Schiffler \cite{ADS}.
% As an important special case, this includes the group of cluster automorphisms studied by Assem, Schiffler, and Schramchenko \cite{ASS} when $\mce$ has trivial coefficients.    
\end{itemize}

We first present these definitions and then discuss a particular example where all the groups are computed and compared to each other and to the quasi-automorphism group.  

\begin{defn}[Cluster modular group {\cite[Definition $2.14$]{FG}}]\label{CMGDefn} Let $\mce$ be a seed pattern with exchange graph~$\mathbf{E}$. The \emph{cluster modular group} $\CMG(\mce)$ is the group of graph automorphisms $g \in \Aut(\mathbf{E})$ that preserve the exchange matrices. More precisely, recall that the unlabeled seed at vertex $t \in \mathbf{E}$ is indexed not by $[1,n]$ but by the elements of $\starr(t)$. Then an element of the cluster modular group is a graph automorphism $g \in \text{Aut}(\mathbf{E})$ satisfying $B(t)_{t',t''} = B(g(t))_{g(t'),g(t'')}$ for all $t \in \mathbf{E}$ and $t',t'' \in \starr(t)$. Such a graph automorphism can be determined by choosing a pair of vertices $t_0,\ov{t_0} \in \mathbf{E}$ and an identification of $\starr(t_0)$ with $\starr(\ov{t}_0)$ under which $B(t_0) = B(\ov{t_0})$. 
\end{defn}

\begin{rmk} Because Definition \ref{CMGDefn} is in terms of automorphisms of the exchange graph, the cluster modular group appears to depend on the entire seed pattern $\mce$, and not just the underlying exchange matrices in $\mce$. However, it is widely believed that the exchange graph -- and therefore the cluster modular group -- is in fact independent of the choice of coefficients (i.e., it only depends on the exchange matrices, and therefore can be prescribed by giving a single such matrix). This has been proven for skew-symmetric  exchange matrices \cite{IKLP}.  
\end{rmk}

The quasi-automorphism group is a subgroup of the cluster modular group. Indeed, each quasi-automorphism $\Psi$ determines a cluster modular group element $g$ via~$\Psi(\Sigma(t)) \sim \Psi(g(t))$, and proportional quasi-automorphisms determine the same $g$. Since $\Psi$ preserves exchange matrices and evaluating $\Psi$ commutes with permuting the cluster variables in a seed, the element $g$ produced this way is indeed an element of the cluster modular group.

One can also consider automorphisms in the category of cluster algebras defined in \cite{ADS}. We reproduce a version of the definition for the sake of convenience.  
\begin{defn}\label{strongautomorphism} Let $\mce$ be a seed pattern. We say two seeds $\Sigma_1$ and $\Sigma_2$ in $\mce$ are similar if $\Sigma_2$ coincides with $\Sigma_1$ after first permuting the frozen variables, and then permuting the indices $[1,n]$ appropriately. 

Suppose the exchange matrices in $\mce$ are indecomposable. Let $\mca$ be its cluster algebra. A $\bbz$-algebra map $f \colon \mca \to \mca$ is an \emph{automorphism} of $\mce$ if for every (equivalently, for any) seed $\Sigma$ in $\mce$, $f(\Sigma)$ or $f(\Sigma)^{\text{opp}}$ is similar to a seed in $\mce$. We denote the group of automorphisms of $\mce$ by $\Aut(\mce)$.  
\end{defn}

The elements of $\Aut(\mce)$ are similar to \emph{strong isomorphisms} from \cite{CAII} but slightly more general since one is allowed to permute the frozen variables. 

We say $f$ as in Definition \ref{strongautomorphism} is a \emph{direct automorphism} or \emph{inverse automorphism} according to whether $f(\Sigma)$ or $f(\Sigma)^{\text{opp}}$ is a seed in $\mce$. Let $\Aut^+(\mce) \subset \Aut(\mce)$ denote the subgroup of direct automorphisms. By similar reasoning to~\cite[Theorem 2.11]{ASS}, this subgroup has index two  in $\Aut(\mce)$ if each seed $\Sigma$ in $\mce$ is mutation-equivalent to $\Sigma^{\text{opp}}$; otherwise $\Aut^+(\mce) = \Aut(\mce)$.

An important special case of Definition \ref{strongautomorphism} is when $\mce$ has trivial coefficients in which case the group $\Aut(\mce)$ is the group of \emph{cluster automorphisms} \cite{ASS}. When $\mce$ has trivial coefficients, we have $\Aut^+(\mce) = \CMG(\mce)$. Furthermore, a direct cluster automorphism is the same as a quasi-automorphism in this case.   

\medskip
We can summarize the containments between the preceding groups as 
\begin{equation}\label{containmentseq}
\Aut^+(\mce) \subset \QAut_0(\mce) \subset \CMG(\mce) \overset{?}{=} \Aut^+(\mce_\triv).
\end{equation} 
where $\mce$ is a seed pattern and $\mce_\triv$ is the seed pattern obtained from $\mce$ by trivializing its coefficients. The equality~$\CMG(\mce) \overset{?}{=} \Aut^+(\mce_\triv)$ depends on the belief that $\CMG(\mce) = \CMG(\mce_\triv)$, cf.~Remark~\ref{CMGandCoeffs}. The group $\Aut(\mce_\triv)$ contains all of the groups in \eqref{containmentseq}, and the group $\Aut(\mce)$ doesn't sit nicely with the rest of the containments when $\Aut^+(\mce) \subsetneq \Aut(\mce)$. 

We next illustrate the differences between the groups in \eqref{containmentseq} using a particular cluster algebra associated with a bordered marked surface. Basic notions and references concerning this class of cluster algebras are given in Section \ref{SurfacesSecn}.  

\begin{example}\label{annulusclustereg} Let $(\mathbf{S,M})$ be an annulus with two marked points on each boundary component cf.~Figure~\ref{annulusfig}. We have colored the marked points either black or white to aid in describing the automorphism groups below. 

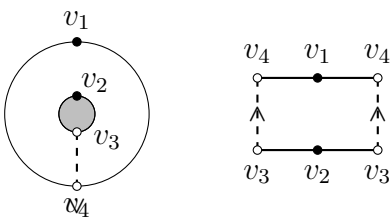
\begin{figure}[ht]
\begin{center}
\begin{tikzpicture}[scale = .8]
\def  \rsizeo{.3};
\def  \rsizet{1.2};
\draw [fill= lightgray] (0,0) circle [radius = \rsizeo];
\draw (\rsizeo,0) arc [radius = \rsizeo , start angle = 0, end angle = 360];
\draw (\rsizet,0) arc [radius = \rsizet , start angle = 0, end angle = 360];
\draw [thick, dashed] (0,-\rsizeo)--(0,-\rsizet);
\node at (0,-.7^\rsizet) {$\vee$};
\node at (0,1.3*\rsizet) {$v_1$};
\node at (0,-1.3*\rsizet) {$v_4$};
\node at (60:2*\rsizeo) {$v_2$};
\node at (-40:2.2*\rsizeo) {$v_3$};
\draw [fill= white] (0,-\rsizet) circle [radius = .07];
\draw [fill= white] (0,-\rsizeo) circle [radius = .07];
\draw [fill= black] (0,\rsizet) circle [radius = .07];
\draw [fill= black] (0,\rsizeo) circle [radius = .07];

\begin{scope}[yshift = -.6cm, xshift = 3cm]
\node at (0,.6) {$\wedge$};
\node at (2,.6) {$\wedge$};
\draw [thick,dashed] (0,0)--(0,1.2);
\draw [thick,dashed] (2,0)--(2,1.2);
\draw [thick] (0,0)--(2,0);
\draw [thick] (0,1.2)--(2,1.2);
\draw [fill= white] (0,0) circle [radius = .07];
\draw [fill= white] (0,1.2) circle [radius = .07];
\draw [fill= white] (2,0) circle [radius = .07];
\draw [fill= white] (2,1.2) circle [radius = .07];
\draw [fill= black] (1,0) circle [radius = .07];
\draw [fill= black] (1,1.2) circle [radius = .07];
\node at (0,1.6) {$v_4$};
\node at (2,1.6) {$v_4$};
\node at (1,1.6) {$v_1$};
\node at (0,-.4) {$v_3$};
\node at (2,-.4) {$v_3$};
\node at (1,-.4) {$v_2$};
\end{scope}
\end{tikzpicture}
\caption{An annulus with two marked points on each boundary component. At right, we show a ``flat form'' of this annulus obtained by cutting along the dashed line. \label{annulusfig}}
\end{center}
\end{figure}

The cluster modular group $\CMG(\mathbf{S,M})$ for a cluster algebra associated with this annulus coincides with the \emph{mapping class group} of the annulus (see Proposition~\ref{ASSconjecture} below).  
%as follows from \cite[Theorem $4.12$]{ASS}. 
This group has the following explicit description: let $\rho$ be the (isotopy class of) the homeomorphism of $\mathbf{S}$ that rotates the inner boundary of the annulus clockwise by a half-turn. Let $\tau$ be the clockwise half-turn of the outer boundary. Let $\sigma$ be the homeomorphism represented by a $180$ degree turn of the flat form of the annulus; it swaps the inner and outer boundary components. Then the elements $\rho,\tau,$ and $\sigma$ generate the cluster modular group.  The group has a presentation $\CMG(\mathbf{S,M}) = \langle \rho,\tau,\sigma \colon (\rho \tau)^2 = \sigma^2 = 1,  \rho \tau = \tau \rho, \sigma \rho = \tau \sigma \rangle$ with respect to these generators. It is a central extension $1 \mapsto \bbz / 2 \bbz \mapsto \CMG \mapsto \Dih_\infty \mapsto 1$ of the infinite dihedral group $\Dih_\infty = \langle r,s \colon s^2 = (sr)^2 = 1 \rangle $ by~$\bbz  / 2 \bbz = \langle  \rho \tau \rangle$, using  the map $\sigma \mapsto s, \rho \mapsto r, \tau \mapsto r^{-1}$.

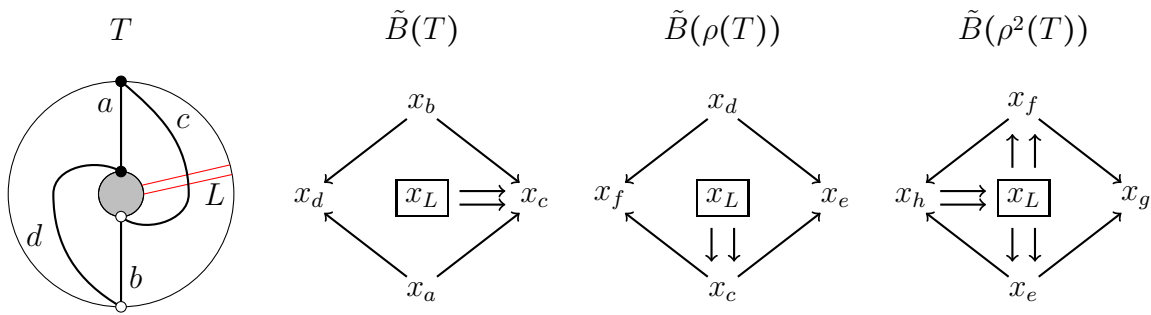
\begin{figure}
\begin{tikzpicture}[scale = 1]
\def  \rsizeo{.3};
\def  \rsizet{1.5};
\node at (0,2.2) {$T$};
\draw [fill= lightgray] (0,0) circle [radius = \rsizeo];
\draw (\rsizeo,0) arc [radius = \rsizeo , start angle = 0, end angle = 360];
\draw (\rsizet,0) arc [radius = \rsizet , start angle = 0, end angle = 360];
\draw [red] (\rsizeo,0)--(10:\rsizet);
\draw [red] (23:\rsizeo)--(15:\rsizet);
\draw [thick] (0,\rsizeo)--(0,\rsizet);
\draw [thick] (0,-\rsizeo)--(0,-\rsizet);
\draw [thick] (-90: \rsizeo) to [out = -40, in = 270] (0:.6*\rsizet) to [out = 90, in = -40] (90: \rsizet);
\draw [thick] (90: \rsizeo) to [out = 150, in = 90] (180:.6*\rsizet) to [out = 270, in = 150] (-90: \rsizet);
\node at (100:.8*\rsizet) {$a$};
\node at (-80:.75*\rsizet) {$b$};
\node at (50:.85*\rsizet) {$c$};
\node at (-155:.85*\rsizet) {$d$};
\node at (0:.83*\rsizet) {$L$} ; 
\draw [fill= white] (0,-\rsizet) circle [radius = .07];
\draw [fill= white] (0,-\rsizeo) circle [radius = .07];
\draw [fill= black] (0,\rsizet) circle [radius = .07];
\draw [fill= black] (0,\rsizeo) circle [radius = .07];

\begin{scope}[xshift = 4cm, yshift = -1.3cm]
\node at (0,3.5) {$\tilde{B}(T)$};
\node at (0,0) {$x_a$};
\node at (0,2.5) {$x_b$};
\node at (-1.5,1.25) {$x_d$};
\node at (1.5,1.25) {$x_c$};
\node at (0,1.25) {$\boxed{x_L} $};
\draw [thick, ->] (.2,.2)--(1.3,1.05);
\draw [thick, ->] (-.2,.2)--(-1.3,1.05);
\draw [thick, ->] (.2,2.3)--(1.3,1.45);
\draw [thick, ->] (-.2,2.3)--(-1.3,1.45);
\draw [thick, ->] (.5,1.16)--(1.11,1.16);
\draw [thick, ->] (.5,1.34)--(1.11,1.34);
\end{scope}

\begin{scope}[xshift = 8cm, yshift = -1.3cm]
\node at (0,3.5) {$\tilde{B}(\rho(T))$};
\node at (0,0) {$x_c$};
\node at (0,2.5) {$x_d$};
\node at (-1.5,1.25) {$x_f$};
\node at (1.5,1.25) {$x_e$};
\node at (0,1.25) {$\boxed{x_L} $};
\draw [thick, ->] (.2,.2)--(1.3,1.05);
\draw [thick, ->] (-.2,.2)--(-1.3,1.05);
\draw [thick, ->] (.2,2.3)--(1.3,1.45);
\draw [thick, ->] (-.2,2.3)--(-1.3,1.45);
\draw [thick, ->] (-.15,.85)--(-.15,.4);
\draw [thick, ->] (.15,.85)--(.15,.4);
\end{scope}

\begin{scope}[xshift = 12cm, yshift = -1.3cm]
\node at (0,3.5) {$\tilde{B}(\rho^2(T))$};
\node at (0,0) {$x_e$};
\node at (0,2.5) {$x_f$};
\node at (-1.5,1.25) {$x_h$};
\node at (1.5,1.25) {$x_g$};
\node at (0,1.25) {$\boxed{x_L} $};
\draw [thick, ->] (.2,.2)--(1.3,1.05);
\draw [thick, ->] (-.2,.2)--(-1.3,1.05);
\draw [thick, ->] (.2,2.3)--(1.3,1.45);
\draw [thick, ->] (-.2,2.3)--(-1.3,1.45);
\draw [thick, ->] (-.15,.85)--(-.15,.4);
\draw [thick, ->] (.15,.85)--(.15,.4);
\draw [thick, ->] (-.15,1.65)--(-.15,2.1);
\draw [thick, ->] (.15,1.65)--(.15,2.1);
\draw [thick, ->] (-1.11,1.16)--(-.5,1.16);
\draw [thick, ->] (-1.11,1.34)--(-.5,1.34);
\end{scope}

\end{tikzpicture}
\caption{A lamination $L$ consisting of two copies of the same curve on the annulus, determining a single frozen variable $x_L$. We have also drawn a triangulation $T$ of this annulus by the arcs $a,b,c,d$. The quivers $\tilde{B}(T),\tilde{B}(\rho(T))$, and $\tilde{B}(\rho^2(T))$ are shown at right, where the extra arcs are~$e = \rho^2(a), f = \rho^2(b), g = \rho^2(c), h = \rho^2(d)$. The values of $\hat{y}$ in each quiver are read off as the Laurent monomial ``incoming variables divided by outgoing variables.''  
\label{annulusclusterfig}} 
\end{figure}

Figure~\ref{annulusclusterfig} gives a choice of lamination~$L$ and triangulation $T$, as well as the quivers $\tilde{B}(T), \tilde{B}(\rho(T))$ and $\tilde{B}(\rho^2(T))$. Let $\mca$ be the corresponding cluster algebra with frozen variable $x_L$, $\mce$ its seed pattern,  and $\mcf_{>0}$ its ambient semifield.  

The cluster modular group element $\sigma \rho \tau$ permutes the arcs in $T$. It induces an automorphism of the quiver $\tilde{B}(T)$, and therefore an element of $\Aut(\mce)$.   

The quivers $\tilde{B}(T)$ and $\tilde{B}(\rho^2(T))$ are neither isomorphic nor opposite, so there is no strong automorphism sending the seed at $T$ to the seed at $\rho^2(T)$. Likewise, there is no strong automorphism between $T$ and $\rho^{\pm 4}( T), \rho^{\pm 6} (T), \rho^{\pm 8} ( T)$, and so on. However, there \emph{is} a quasi-automorphism relating these seeds, which we describe now. It is the semifield map $\Psi \colon \mcf_{>0} \to \mcf_{>0}$ defined by $\Psi(x_L) = x_L$ as well as $\Psi(x_\gamma) = x_L^{-1} \cdot  x_{\rho^2(\gamma)}$ for $\gamma = a,b,c,d$. It sends each $\hat{y}$ for $\Sigma(T)$ to the corresponding $\hat{y}$ for $\Sigma(\rho^2 \cdot T)$, defining a quasi-automorphism of $\mce$ whose map on seed orbits is $\rho^2$. It has a simple global description on cluster variables which can be checked inductively by performing appropriate mutations away from $T$. Namely, for each arc $\gamma$ let $\iota(\gamma,L)$ denote the number of times $\gamma$ crosses the two curves in $L$. For example, $\iota(a,L) = 0$ and $\iota(c,L) = 1$. Then 
\begin{equation}\label{powerofqforgamma}
\Psi(x_\gamma) = x_L^{\iota(\gamma,L) - \iota(\rho^2(\gamma),L)} \cdot x_{\rho^2( \gamma)}
\end{equation}
for all arcs $\gamma$ in the annulus. The power of $x_L$ on the right hand side of \eqref{powerofqforgamma} is always equal to~$0,1$, or $- 1$. It is also simple to describe quasi-automorphisms realizing $\sigma$ and $\rho \tau$. 

Perhaps surprisingly, the seeds at $T$ and $\rho (T)$ are \emph{not} related by a quasi-automorphism. Indeed, the values of $\hat{y}$ are equal at the top and bottom vertices of $\tilde{B}(T)$ in Figure \ref{annulusclusterfig}, but they are not equal in $\tilde{B}(\rho ( T))$. The same holds for $T$ and $\tau ( T)$. 

Putting all of this together, there is only one nontrivial strong automorphism of $\mce$, namely the element $\sigma \rho \tau$. On the other hand, the quasi-automorphism group is infinite, generated by $\rho^2 , \sigma$ and  $\rho \tau$. It is a direct product $\Dih_\infty \times \bbz / 2 \bbz$. It is an index two subgroup of $\CMG(\mathbf{S,M})$, namely the kernel of the map $\CMG \twoheadrightarrow \bbz  / 2 \bbz$ that computes the parity of the number of black marked points sent to a white marked point.
\end{example}

Example \ref{annulusclustereg} suggests that although the cluster modular group $\CMG(\mce)$ may be strictly larger than the quasi-automorphism group $\QAut_0(\mce)$, the gap between these groups is not so large. Indeed, Section \ref{SurfacesSecn} establishes that for seed patterns associated with surfaces, $\QAut_0(\mce)$ is always a finite index subgroup of the cluster modular group.

\section{Quasi-automorphisms of cluster algebras from surfaces}\label{SurfacesSecn}
In this section we place Example~\ref{annulusclustereg} in context via results valid for any cluster algebra associated to a marked bordered surface as in \cite{ModuliSpaces, CATSI,CATSII, GSV2003}. We describe quasi-automorphisms of these cluster algebras in terms of the tagged mapping class group of the marked surface. 

 We follow the setup and notation in \cite{CATSII}. Let $\mca(\mathbf{S,M,L})$ denote the cluster algebra of geometric type determined by the triple $(\mathbf{S,M,L})$. Here $\mathbf{S}$ is an oriented bordered surface with a nonempty set~$\mathbf{M}$ of \emph{marked points}. The marked points reside either in the interior of $\mathbf{S}$ (we call these \emph{punctures}) or in $\partial \mathbf{S}$ (we call these \emph{cilia}). The set of punctures is $\ov{\mathbf{M}}$. We disallow a few possibilities for $(\mathbf{S,M})$, namely a sphere with three or fewer punctures, an $n$-gon when $n<4$, and a once-punctured monogon. The choice of coefficients is specified by a \emph{multi-lamination} $\mathbf{L} = (L_1,\dots,L_m)$, an $m$-tuple of (integral unbounded measured) laminations on $\mathbf{(S,M)}$. Each lamination $L_i$ consists of a finite number of \emph{curves} in $\mathbf{(S,M)}$.  

The cluster variables in $\mca(\mathbf{S,M})$ are indexed by \emph{tagged arcs}~$\gamma$, the set of which we denote by~$\mathbf{A}^\bowtie(\mathbf{S,M})$. The seeds in~$\mca(\mathbf{S,M,L})$ are indexed by \emph{tagged triangulations}~$T$ of~$(\mathbf{S,M})$. The extended exchange matrix~$\tilde{B}(T)$ for a seed has the \emph{signed adjacency matrix}~$B(T)$ as its principal part, and has the \emph{shear coordinate vector}~$\vec{b}(T,L_i)$ of the lamination~$L_i$ with respect to~$T$ as its~$i^{\text{th}}$ row of coefficients. 

The exchange graph of the resulting cluster algebra is independent of the choice of coefficients \cite[Corollary 6.2]{CATSII}. We let $\CMG(\mathbf{S,M})$ denote the corresponding cluster modular group. It is closely related to the following geometrically defined group.  

\begin{defn}[Tagged mapping class group \cite{ASS}]\label{markedmcgdefn}
Let $(\mathbf{S,M})$ be a bordered marked surface that is not a closed surface with exactly one puncture. 
A \emph{tagged mapping class} for $(\mathbf{S,M})$ is a pair $g = (f,\psi)$, where 
\begin{itemize}
\item $f$ is an element of the \emph{mapping class group} of $(\mathbf{S,M})$ -- i.e. $f$ is an orientiation-preserving homeomorphism of $\mathbf{S}$ mapping $\mathbf{M}$ to itself setwise, considered up to isotopies of $\mathbf{S}$ that fix $\mathbf{M}$ pointwise, and  
\item $\psi \colon \ov{\mathbf{M}} \to \{\pm 1\}$ is a function from the set of punctures to $\{\pm 1\}$.
\end{itemize}
When $(\mathbf{S,M})$ is a closed surface with one puncture $p$, we make the same definition but impose $\psi(p) = 1$ since tagged versions of arcs are not in the cluster algebra. The tagged mapping classes comprise the \emph{tagged mapping class group}, denoted $\mcm \mcg_\bowtie(\mathbf{S,M})$.
\end{defn}

We understand $\mcm \mcg_\bowtie(\mathbf{S,M})$ by its action on tagged arcs $\gamma \in A^\bowtie(\mathbf{S,M})$. A tagged mapping class $g = (f,\psi)$ acts on $\gamma$ by first performing the homeomorphism $f$ to $\gamma$, and then changing the tag of any end of $\gamma$ incident to a puncture $p$ for which $\psi(p) = -1$. 
The resulting action on tagged triangulations preserves the signed adjacency matrices, and embeds $\mcm \mcg_\bowtie(\mathbf{S,M})$ as a subgroup of  $\CMG(\mathbf{S,M})$, cf.~\cite{ASS}. In fact the following is true: 

\begin{prop}[Bridgeland-Smith]\label{ASSconjecture} The tagged mapping class group $\mcm \mcg _\bowtie (\mathbf{S,M})$ coincides with the cluster modular group $\CMG(\mathbf{S,M})$, unless $(\mathbf{S,M})$ is a sphere with four punctures, a once-punctured square, or a digon with one or two punctures. 
\end{prop}  

Thus barring these exceptional cases, two tagged triangulations of $(\mathbf{S,M})$ have isomorphic quivers precisely when they are related by an element of the tagged mapping class group (see \cite[Proposition 8.5]{BridgelandSmith} and the subsequent discussion; see also \cite[Conjecture 1]{ASS}). In the exceptional cases listed in Proposition \ref{ASSconjecture}, the tagged mapping class group is a proper finite index subgroup of the cluster modular group. 

Motivated by Proposition \ref{ASSconjecture}, we set out to describe, for various choices of coefficients~$\mathbf{L}$, the quasi-automorphism group $\QAut_0(\mathbf{S,M,L})$ from Definition \ref{QAutDefn} as a subgroup of the tagged mapping class group. The main ingredient in our answer is a black-white coloring similar to one in the examples from Section \ref{QASecn}. 

\begin{defn}\label{colorCs} The \emph{even components} of $(\mathbf{S,M})$ are the punctures $C \in \ov{M}$ as well the as boundary components $C \subset \partial \mathbf{S}$ having an even number of cilia. 
%We think of a puncture as a shrunk boundary component with $0$ cilia. 
We let $r$ denote the number of even components, and label the even components $C_1,\dots,C_r$. For each even boundary component $C \subset \partial \mathbf{S}$, we color the cilia on $C$ black or white so that the colors alternate, i.e. adjacent cilia  have opposite colors. 
\end{defn}

Using the black-white coloring in Definition \ref{colorCs}, each tagged mapping class $g = (f,\psi)$ determines an $r \times r$ signed permutation matrix $\pi_g$ whose entries are indexed by the even components. The $(i,j)$ entry of $\pi_g$ is $0$ unless $f(C_i) = C_j$. If $C_j \subset \partial \mathbf{S}$ is a boundary component and $f(C_i) = C_j$, then the $(i,j)$ entry is $+1$ if $f$ sends black cilia on $C_i$ to black cilia on $C_j$, and is $-1$ if $f$ sends black cilia on $C_i$ to white cilia on $C_j$. When $C_i$ and $C_j$ are punctures, the sign of the $(i,j)$ entry is the sign $\psi(C_j)$. Not all signed permutation matrices will arise in this way  since $f$ can only permute components that have the same number of cilia.

%This sign is $+1$ if $E$ is on an even boundary component $C$, and its nearest neighboring cilium (traveling clockwise along $C$ from $E$) is black. The sign is $-1$ one if this nearest neighbor is white. If $E$ spirals around a puncture $C$, the sign is $+1$ if $E$ spirals counterclockwise around $C$ and $-1$ otherwise. 

\begin{defn}\label{signconvention}
Let $L$ be a lamination. For each curve $\aa$ in $L$, Figure \ref{colorsofendsfig} shows how to assign a sign to an end of $\aa$ that either lands on even boundary component or spirals around a puncture. At a puncture, the sign is chosen according to whether~$\aa$ spirals counterclockwise or clockwise into the puncture. At a boundary component, the sign is chosen according to whether the nearest neighboring cilium in the clockwise direction along~$C$ is black or white. An end on an odd component has zero sign. The \emph{pairing}~$p( L;C) $ of a lamination~$L$ with the even component~$C$ is the sum of all the signs associated to~$L$, i.e. the sum over all curves $\aa$ in $L$ of the signs of the two ends of $\aa$. We let~$\vec{p}(L) = (p(L;C_i))_{i = 1,\dots,r} \in \bbz^r$ denote the vector of pairings of~$L$ with the even components.% in~$(\mathbf{S,M})$.  
\end{defn}

Example~\ref{workedsignsexample} works out these signs for the annulus from Example~\ref{annulusclustereg}. 

\begin{figure}[ht]
\begin{center}
\begin{tikzpicture}

\def  \rsizeo{.6};
\def  \rsizet{1.5};
\draw [fill= lightgray] (0,0) circle [radius = \rsizeo];
\draw [fill = black] (0:\rsizeo) circle [radius = .07];
\draw [fill = white] (90:\rsizeo) circle [radius = .07];
\draw [fill = black] (180:\rsizeo) circle [radius = .07];
\draw [fill = white] (-90:\rsizeo) circle [radius = .07];
\draw [thick] (20:\rsizeo) to [out = 0, in = -90] (1.8*\rsizeo, 1.3*\rsizeo);
\draw [thick, dashed] (1.8*\rsizeo,1.5*\rsizeo)--(1.8*\rsizeo,2*\rsizeo);
\node at (2.2*\rsizeo,.2) {+1};

\begin{scope}[xshift = 3cm]
\draw [fill= lightgray] (0,0) circle [radius = \rsizeo];
\draw [fill = black] (0:\rsizeo) circle [radius = .07];
\draw [fill = white] (90:\rsizeo) circle [radius = .07];
\draw [fill = black] (180:\rsizeo) circle [radius = .07];
\draw [fill = white] (-90:\rsizeo) circle [radius = .07];
\draw [thick] (-35:\rsizeo) to [out = -10, in = -90] (1.8*\rsizeo, 1.3*\rsizeo);
\draw [thick, dashed] (1.8*\rsizeo,1.5*\rsizeo)--(1.8*\rsizeo,2*\rsizeo);
\node at (2.3*\rsizeo,.2) {-1};
\end{scope}

\begin{scope}[xshift = 9cm]
\node at (0,0) {$\bullet$};
\draw [thick] (0:1.2*\rsizeo) to [out = -90, in = 0] (-90:\rsizeo) to [out = 180, in = -90] (180:\rsizeo) to [out = 90, in = 180] (90:\rsizeo)
to [out = 0, in = 90] (0:.8*\rsizeo) to [out = -90, in = 0] (-90:.7*\rsizeo) to [out = 180, in = -90] (180:.7*\rsizeo);
\draw [thick, dashed] (0:1.2*\rsizeo)--(45:1.4*\rsizeo)--(55:1.6*\rsizeo)--(60:1.9*\rsizeo);
\draw [thick, dashed] (180:.7*\rsizeo) to [out = 90, in = 180] (90:.7*\rsizeo) to [out = 0, in = 90] (0:.55*\rsizeo);
\node at (2.2*\rsizeo,.2) {-1};
\end{scope}

\begin{scope}[xshift = 6cm]
\node at (0,0) {$\bullet$};
\draw [dashed,thick] (140:1.4*\rsizeo)--(120:2*\rsizeo);
\draw [thick] (145:1.4*\rsizeo) to [out =-90 , in = 90](180:1.2*\rsizeo) to [out = -90, in = 180] (-90:\rsizeo) to [out = 0, in = -90] (0:\rsizeo) to [out = 90, in = 0] (90:\rsizeo) to [out = 180, in = 90] (180:.9*\rsizeo) to [out = -90, in = 180] (-90:.75*\rsizeo) 
to [out = 0, in = -90] (0:.7*\rsizeo) to [out = 90, in = 0] (90:.6*\rsizeo) to [out = 180, in = 90] (180:.5*\rsizeo);
\draw [dashed, thick] (180:.5*\rsizeo) to [out = -90, in = 180] (-90:.5*\rsizeo) to [out = 0, in = -90](0:.4*\rsizeo);
\node at (1.6*\rsizeo,.2) {+1};
\end{scope}

\end{tikzpicture}
\caption{The conventions for assigning signs to each end of a curve that lands on an even boundary component (in this case, a boundary component with $4$ cilia) or spirals around a puncture. The pairing $p(L,C)$ is obtained by adding up all of these signs. \label{colorsofendsfig}}
\end{center}
\end{figure}
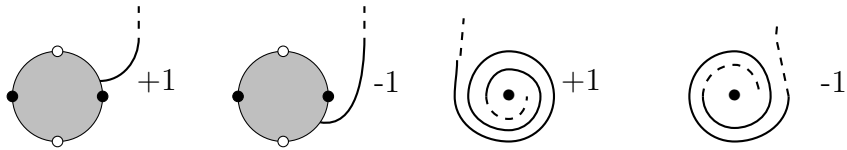

In addition to acting on tagged arcs, $\mcm \mcg_\bowtie(\mathbf{S,M})$ also acts on laminations $L$. A tagged mapping class $g = (f,\psi)$ acts by first performing the homeomorphism $f$ to $L$, and then changing the direction of spiral at each puncture $p$ for which $\psi(p) = -1$. This action preserves sheard coordinates in the sense that $\vec{b}(T,L) = \vec{b}(g(T),g(L))$ for a triangulation $T$ and lamination $L$. It is easy to see that $g$ acts on the vector of pairings by the matrix $\pi_g$, i.e. $\vec{p}(g (L)) = \pi_g \cdot \vec{p}(L)$ for a lamination $L$.

The next theorem is the main result of this section, describing $\QAut(\mathbf{S,M,L})$ inside the marked mapping class group in very concrete terms. 
\begin{thm}\label{whichgworkwithL} Suppose $(\mathbf{S,M})$ is not one of the four exceptional surfaces in~Proposition \ref{ASSconjecture}. Let $\mathbf{L}$ be a multi-lamination. Let $V_\mathbf{L} = \Span(\{\vec{p}(L) \colon L \in \mathbf{L}\}) \subset \bbz^r$ be the submodule spanned by the vectors of pairings associated to the laminations $L $ in $\mathbf{L}$. Then 
\begin{equation}\label{whichgworkeq}
\QAut_0(\mathbf{S,M,L}) = \{g \in \mcm \mcg_\bowtie(\mathbf{S,M}) \colon \pi_g(V_\mathbf{L})  = V_\mathbf{L}\}.  
\end{equation}
\end{thm}

We prove Theorem \ref{whichgworkwithL} in Section \ref{SurfacesProofsSecn}. The subgroup of $\mcm \mcg_\bowtie(\mathbf{S,M})$ described in \eqref{whichgworkeq} only depends on the endpoint behavior of laminations -- it doesn't mention the topology of the surface, or how much curves wrap around the holes and handles of the surface. The map $g \mapsto \pi_g$ is a group homomorphism from $\mcm \mcg_\bowtie(\mathbf{S,M})$ to the group of signed permutation matrices. The subgroup in \eqref{whichgworkeq} is an inverse image of the subgroup of signed permutation matrices that fix $V_\mathbf{L}$ and therefore is always finite index in $\mcm \mcg_\bowtie(\mathbf{S,M})$. 

\begin{cor}\label{QAutBigcapCor}
Let $g$  be a tagged mapping class. If $\pi_g$ is plus or minus the identity matrix, then $g \in \QAut_0(\mathbf{S,M,L})$ for \emph{any choice} of multi-lamination $\mathbf{L}$. Otherwise, $g \notin \QAut_0(\mathbf{S,M,L})$ for \emph{some choice} of~$\mathbf{L}$. 
\end{cor}

\begin{rmk} The tagged mapping classes in Corollary \ref{QAutBigcapCor} are those that fix all even components setwise, and furthermore either preserve the black-white coloring of ends of curves, or simultaneously swap all colors. This group is generated by the following four types of elements (see~\cite{FarbMargalit} for generators of the mapping class group): Dehn twists about simple closed curves, homeomorphisms that permute odd components, fractional Dehn twists rotating the cilia on a given  boundary component by two units, and the \emph{tagged rotation}. This last element is the one that simultaneously changes tags at all punctures and rotates all boundary components by one unit. It was studied in  \cite{BrustleQiu}, where it was shown to coincide with the shift functor of a $2$-Calabi-Yau cluster category associated with the surface. 
\end{rmk}

\begin{proof}
If $\pi_g = \pm 1$, then $\pi_g$ clearly preserves $V_\mathbf{L}$ regardless of the choice of $\mathbf{L}$ and by Theorem \ref{whichgworkwithL} $g$ is in $\QAut_0(\mathbf{S,M,L})$ for any $\mathbf{L}$. 

If $\pi_g \neq \pm 1$, think of $\pi_g$ as a signed permutation $\sigma$ of $\{\pm 1,\dots,\pm r\}$ in the usual way. If there is any index $i \in [1,r]$ such that $\sigma(i) \neq \pm i$, then let $L$ be a lamination consisting of a curve with two black ends on $C_i$, satisfying $p(L;C)  = 2$. If there is no such index $i$, we can choose a pair of indices $i , j \in [1,r]$ such that $\sigma(i) = -i$ but $\sigma(j) = j$. In this case we let $L$ be a lamination consisting of a curve connecting the even components $C_i$ and $C_j$ by a curve that is black at both ends. In both of these two cases, we see that
$\pi_g (\vec{p}(L))$ is not in the span of $\vec{p}(L)$ and by Theorem \ref{whichgworkwithL}, $g \notin \QAut_0(\mathbf{S,M,L})$.  
\end{proof}

\begin{example}\label{workedsignsexample} We order the boundary components in Figure \ref{annulusclusterfig} so that the inner boundary is first. The vector of pairings for the lamination $L$ in Figure \ref{annulusclusterfig} is $\vec{p}(L)  = (-2,-2)$. Then $\rho$ and $\tau$ act on the vector of parings by swapping the sign of the first or second component respectively, and $\sigma$ acts by permuting the first and second component. The description of $\QAut_0(\mathbf{S,M,L})$ in Example \ref{annulusclustereg} matches the one in Theorem \ref{whichgworkwithL}. The subgroup of elements described in Corollary \ref{QAutBigcapCor} is a
a direct product $\bbz \times \bbz / 2 \bbz$ generated by 
$\rho^2$ and  $\rho \tau$. It has index $4$ in the cluster modular group.    
\end{example}

\begin{rmk} Theorem \ref{whichgworkwithL} can be modified in the case that $(\mathbf{S,M})$ is one of the exceptional surfaces in Proposition \ref{ASSconjecture}. Namely, the left hand side of \eqref{whichgworkeq} merely describes the subgroup of $\QAut_0(\mathbf{S,M,L})$ consisting of tagged mapping classes (that is, ignoring the exotic symmetries). For particular choices of coefficients, the ``extra'' elements of the cluster modular group might also be inside $\QAut_0$. 
\end{rmk}

\section{Proofs for Section \ref{SurfacesSecn}}\label{SurfacesProofsSecn}
The key result of this section is Proposition \ref{QHsofSurfaces} describing quasi-homomorphisms of cluster algebras from surfaces. Theorem \ref{whichgworkwithL} follows from it as a special case. 

Let $\mathbf{B}(\mathbf{S,M})$ denote the set of \emph{boundary segments} connecting adjacent cilia in $\partial S$. There is an especially natural choice of multi-lamination $\mathbf{L}_{\text{boundary}} = (L_\bb)_{\bb \in \mathbf{B(S,M)}}$ with one frozen variable for each boundary segment (see e.g. in \cite[Remark 15.8]{CATSII}). Lemma \ref{sccshearcoord} expresses the shear coordinates of certain laminations in terms of the extended exchange matrix determined by $\mathbf{L}_{\text{boundary}}$. It is patterned after \cite[Lemma $14.3$]{CATSII}.

\begin{lem}\label{sccshearcoord} Let $L$ be a lamination none of whose curves has an end that spirals at a puncture. Given an arc $\gamma$ the \emph{transverse measure} of $\gamma$ in $L$ is the minimal number of intersections of $\gamma$ with the curves in $L$. We denote it by $l(\gamma,L)$. For a boundary segment $\bb \in \mathbf{B}(\mathbf{S,M})$, we similarly let $l(\bb,L)$ denote the number of ends of the curves in $L$ on $\bb$. We let $\vec{l}(T,L) = (l(\star,L))_{\star \in T \cup \mathbf{B}(\mathbf{S,M})}$ be the row vector containing all of these transverse measures. 
Then 
\begin{equation}\label{sccshearcoordeq}
-2 \vec{b}(T,L) = \vec{l}(T,L) \tilde{B}(T,\mathbf{L}_\text{boundary}). 
\end{equation}
\end{lem} 
 
\begin{proof} We check that the $\gamma_0$ components of the left and right hand sides of \eqref{sccshearcoord} are equal, where $\gamma_0 \in T$.  Let $\gamma_1,\dots,\gamma_4$ be the quadrilateral containing $\gamma_0$ (number in clockwise order). Some of the $\gamma_i$ may be boundary segments. Each time $\aa$ shears across the quadrilateral in an `$S$' crossing, it contributes $+1$ to the left hand side, while contributing $-\frac{1}{2}(-1+-1)$ to the right hand side. And so on.    
\end{proof}

This argument is like \cite[Lemma 14.13]{CATSII} but simpler because our we are not dealing with spirals at the puncture, for which $l(T,L) = \infty$. Comparing \eqref{sccshearcoordeq} with \eqref{newbtilde}, we see that if $\mathbf{L}$ is any multi-lamination none of whose curves spiral at punctures, then there is a quasi-homomorphism from $\mca(\mathbf{S,M},\mathbf{L}_\text{boundary})$ to $\mca(\mathbf{S,M,L})$. A version of Lemma \ref{sccshearcoord} allowing for spirals at punctures would involve the extended exchange matrix $\tilde{B}(\ov{T},\ov{L})$ on the \emph{fully opened surface} $(\ov{\mathbf{S,M}})$, where $\ov{L}$ and $\ov{T}$ are lifts of $L$ and $T$ to the opened surface (see \cite[Sections 9,14]{CATSII} for details). The corresponding version of \eqref{sccshearcoordeq} determines a quasi-homomorphism from $\mca(\ov{\mathbf{S,M}},\ov{\mathbf{L}})$ to $\mca(\mathbf{S,M,L})$. 

\medskip

Our next step is to to describe the row span of the signed adjacency matrices $B(T)$. It strengthens \cite[Theorem 14.3]{CATSI} which states that the corank of $B(T)$ is the number of even components. The description requires associating a sign to the ends of arcs $\gamma \in A^\bowtie(\mathbf{S,M})$, in a similar fashion as was done for the ends of curves in Definition \ref{signconvention}. Namely if $\gamma$ has an endpoint on a boundary component $C \subset \partial \mathbf{S}$, the endpoint gets sign $\pm 1$ if its endpoint is a black or white cilium respectively. If the endpoint is on a puncture $C \in \ov{M}$, the sign is $\pm 1$ if the end is plain or tagged respectively. An endpoint on an odd component gets sign $0$. The pairing $p(\gamma;C)$ of $\gamma$ with $C$ is the sum of the signs of the ends of $\gamma$ that reside on $C$, and the vector of pairings is $\vec{p}(\gamma) = (p(\gamma; C_i))_{i = 1, \dots, r}$. This pairing satisfies $p(\gamma;C) = p(L_\gamma;C)$ where $L_\gamma$ is the \emph{elementary lamination} determined by $\gamma$ (cf.\ \cite[Definition 17.2]{CATSII}, see also \cite[Section 5]{ReadingSurfaces}).

\begin{lem}\label{kernelbasisatClem}
For a tagged triangulation $T$, let $\bbq(T) = \bbq(\gamma \colon \gamma \in T)$ be $\bbq$-vector space of row vectors with entries indexed by $\gamma \in T$. Let $\bbq(T)^* = \bbq(\gamma^* \colon \gamma \in T)$ be the dual space of column vectors with entries indexed by the dual basis $\{\gamma^* \colon \gamma \in T \} $. 
%The submodule $\bbq(T)B(T) \subset \bbq(T)$ is the row span of $B(T)$. 

For each even component $C$, consider the column vector
\begin{equation}\label{kernelbasisatC}
R_C = \sum_{\gamma \in T }p(\gamma;C)\gamma^* \in \bbq(T)^*. 
\end{equation} 

Then a vector $\vec{a} \in \bbq(T)$ is in the the row span of $B(T)$ if and only if the dot product $\vec{a} \cdot R_C $ vanishes for all $C$. %That is, the $\{\overset{\downarrow}{R}_C \colon \text{ $C$ even }\}$ is a basis for the dual space to $\bbq(T)B(T) \subset \bbq(T)$ inside $\bbq(T^*)$. 
\end{lem}

Said differently, the map $\gamma \mapsto (\gamma,C_i)$ for $C_i$ an even component determines a $\bbz$-grading on the coefficient-free cluster algebra $\mca(\mathbf{S,M})$. These gradings form a \emph{standard $\bbz^r$-grading} as $i$ varies from $1,\dots,r$ (a standard grading is one that spans the kernel of the $B$-matrices, see \cite{Grabowski}). We will not rely on gradings in what follows. Lemma \ref{kernelbasisatClem} is proved at the end of this section.

%Said differently, the map $\gamma \mapsto (\gamma,C_i)_{i=1}^r$ determines a \emph{standard} $\bbz^{r}$ grading on $\mca(\mathbf{S,M})$ cf.~\cite{Grabowski} and Remark \ref{propandgradings}. The word \emph{standard} means that the rows of the corresponding grading matrix form a basis for the left kernel of $B$.  We will not rely on the interpretation in terms of gradings in what follows. We suspect Lemma \ref{kernelbasisatClem} is familiar to some readers (it is partially hinted at in \cite[Remark 14.8]{CATSII}), so we defer its proof to the end of this section. 

\vspace{.3cm}

For a vector $\vec{a} \in \bbq(T)$, the \emph{residue of $\vec{a}$ around $C$ in $T$} is the dot product $\vec{a} \cdot R_C = \sum_{\gamma \in T }(\gamma,C) a_\gamma$. Writing $\vec{a}$ as the shear coordinate of a lamination $L$, the residue has the following simple description.   
  
\begin{lemma}\label{annularnbhdlemma} 
Let $L$ be a lamination and $C$ an even component. Then the residue of $\vec{b}(T,L)$ around $C$ is the pairing $p(L;C)$ from Definition~\ref{signconvention}.
\end{lemma}

\begin{proof}
The residue is computed in terms of the shear coordinates of arcs adjacent to $C$. To compute these shear coordinates, rather than considering the entire surface, we can focus on the set of triangles having at least one vertex on $C$. Lifting to a finite cover of $\mathbf{S}$ perhaps (in order to remove interesting topology nearby $C$ that is irrelevant to computing the residue) this union of triangles will either be a triangulated annulus (when $C \subset \partial \mathbf{S}$ is a boundary component) or a once-punctured $n$-gon for some $n$ (if $C$ is a puncture). We call this set of triangles the \emph{annnular neighborhood} of $C$. Even when $L$ consists of a single curve, the intersection of $L$ with this annular neighborhood might consist of several curves. By the linearity of shear coordinates and residues, it suffices to consider the case that $L$ consists of a single curve in the annular neighborhood. 

When $C$ is a puncture, its annular neighborhood is a punctured disc with a triangulation all of whose arcs are radii joining the puncture to the boundary of the disc. 
By inspection, a curve $L$ contributes nonzero residue at $C$ if and only it spirals at $C$, and the value of this residue is $\pm 1$ according to whether it spirals counterclockwise or clockwise respectively as claimed. 

When $C$ is a even boundary component, we compute the residue of $\vec{b}(T,L)$ using the right hand side of \eqref{sccshearcoord}. We split up this right hand side into two terms by splitting up $\vec{l}(T,L)$ as a concatenation of $(l(\gamma,L)_{\gamma \in T}$ and  $(l(\beta,L)_{\beta \in \partial \mathbf{S}}$, and performing the matrix multiplication with $\tilde{B}$ in block form. The first term in this expression has zero residue around $C$ since it is a linear combination of the rows of $B(T)$. What's left over is a sum
\begin{equation}\label{residueinannulus}
-\frac{1}{2}\sum_{\gamma \in T, \bb \in \mathbf{B(S,M)} } p(\gamma;C)l(\bb,L) B_{\bb,\gamma}. 
\end{equation}

We claim the sum above evaluates to $p(L; C)$. For the sum to be nonzero, $L$ must have a nonzero end at some segment $\bb = [v_{i-1},v_{i}]$ with $v_{i-1},v_i$ in clockwise order. This segment $\bb$ is contained in a unique triangle in $T$. Call the other two sides in this triangle $\gamma_{i-1}$ and $\gamma_i$, whose endpoint on $C$ is $v_{i-1}$ and $v_i$ respectively. There are cases according to whether either of these sides is a boundary segment. If neither is, then $p(\gamma_{i-1}; C) B_{\bb,\gamma_{i-1}} = p(\gamma_{i}; C)B_{\bb,\gamma_{i}} $ is $\pm 1$ according to whether $v_i$ is white or black. The total contribution to \eqref{residueinannulus} is $ p(C;L)$. In the degenerate case that  $\gamma_i$ is a boundary segment, it does not contribute to \eqref{residueinannulus}, but $p(\gamma_{i-1}; C) = 2$ and this effect is cancelled out, and so on. 
\end{proof}

\begin{prop}\label{QHsofSurfaces} Suppose $(\mathbf{S,M})$ is not among the four listed exceptions in Proposition~\ref{ASSconjecture}. Let $\mathbf{L},\mathbf{L'}$ be multi-laminations on $(\mathbf{S,M})$ and recall the submodules $V_\mathbf{L}$ and $V_\mathbf{L'}$ from Theorem~\ref{whichgworkwithL}. Let $g \in \mcm \mcg_\bowtie(\mathbf{S,M})$ be a tagged mapping class and $\pi_g$ the corresponding signed permutation matrix. The following are equivalent: 
\begin{itemize}
\item there is a quasi-homomorphism $\Psi$ from $\mca(\mathbf{S,M,L})$ to $\mca(\mathbf{S,M,L'})$ whose map on tagged triangulations is $T \mapsto g (T)$ (that is, $\Psi(\Sigma(T)) \sim \ov{\Sigma}(g(T))$),
\item $V_{\mathbf{L'}} \subset \pi_g(V_{\mathbf{L}})$.
\end{itemize}
\end{prop}

\begin{proof}[Proof of Proposition \ref{QHsofSurfaces}] A quasi-homomorphism $\Psi$ from $\mca(\mathbf{S,M,L})$ to $\mca(\mathbf{S,M,L'})$ is determined by a pair of tagged triangulations $T$ and $T'$, such that $\tilde{B}(T,\mathbf{L})$ and $\tilde{B}(T',\mathbf{L'})$ are related as in \eqref{newbtilde}. Since the principal parts of these matrices agree, by Proposition~\ref{ASSconjecture} there is a tagged mapping class $g$ such that $g(T) = T'$. Furthermore, for each lamination $L' \in \mathbf{L'}$, the vector $\vec{b}(T',L')$ must be in $\Span(\{\vec{b}(T,L) \colon L \in \mathbf{L}\})$. Since $\vec{b}(T',L') = \vec{b}(T,g^{-1} (L'))$, by Lemma \ref{kernelbasisatClem}, it is equivalent to find a linear combination of $\vec{b}(T,g^{-1}(L'))$ and $\{\vec{b}(T,L) \colon L \in \mathbf{L}\}$ that has zero residue around every even component. Proposition~\ref{QHsofSurfaces} follows now from Lemma \eqref{annularnbhdlemma} and the fact that $g$ acts on a vector of pairings by the matrix $\pi_g$. 
\end{proof}

\begin{proof}[Proof of Lemma \ref{kernelbasisatClem}]
Restating the Lemma, we seek to show that the $R_C$ form a basis for the dual space to the row span.% $\bbq(T)B(T)$.  %as a subspace of $\bbq(T^*)$. 
We begin by verifying each of these vectors pair to zero with the row span. 

First, we check this when $C$ is a puncture. We begin with the case that all of the arcs in $T$ are untagged at $C$. We need to check that $
\sum_{\gamma \in T} p(\gamma;C) B(T)_{\gamma',\gamma}$ vanishes for each $\gamma' \in T$. Indeed, letting $L $ be the lamination consisting of a tiny simple closed curve contractible to $C$, the shear coordinate vector $\vec{b}(T,L)$ is clearly  $0$. Now we apply \eqref{sccshearcoord} for this choice of $L$: the $\gamma'^{\text{th}}$ component of \eqref{sccshearcoord} says $ 0 = \sum_{\gamma \in T}p(\gamma ; C) B(T)_{\gamma,\gamma'}$ as desired, using the fact that $l(\gamma,L) = 0$ if $\gamma$ is a boundary segment. The argument when all arcs are tagged at $C$ is identical. If $C$ is incident to exactly two arcs, namely the plain and tagged version of the same arc, then $R_C$ follows from \cite[Definition $9.6$]{CATSI} (or a calculation in a once-punctured digon). 

Second we check this when $C \subset \partial \mathbf{S}$ is a boundary component. Number the cilia on $C$ by $v_1, \dots ,v_{2m}$. For each $i \in [1,2m]$, let $L_i$ be a tiny lamination contractible to $v_i$ -- its two endpoints are on the two boundary segments adjacent to $v_i$. Again, $\vec{b}(T,L_i)$ is clearly $0$ and in particular $\sum_{\text{$v_i$ black}} \vec{b}(T,L_i) = \sum_{\text{$v_i$ white}} \vec{b}(T,L_i)$. Summing over the corresponding right hand sides of \eqref{sccshearcoord}, again performing the matrix multiplication in \eqref{sccshearcoord} in block form as in the argument for Lemma \ref{annularnbhdlemma}, the terms corresponding to boundary segments are present in both the sum over black $v_i$ and the sum over white $v_i$. Canceling these common terms, we get the equality $\sum_{\text{$v_i$ black, $\gamma \in T$ }} l(\gamma,L_i) B(T)_{\gamma,\gamma'}= \sum_{\text{$v_i$ white}} l(\gamma,L_i) B(T)_{\gamma,\gamma'}$ for all $\gamma'$, which says $\sum_{\gamma \in T} p(\gamma,C) B(T)_{\gamma,\gamma'} = 0$ for all $\gamma'$ as desired. 

Thus all of the $R_C$ pair to zero with the row span of $B(T)$. We will now show that they are linearly independent, which completes the proof since they have the expected size by  \cite[Theorem $14.3$]{CATSI}. 
  
Consider a linear relation of the form 
\begin{equation}\label{relationofrelations}
\sum a_C R_C = 0. 
\end{equation}
We define scalars $a_v$ for all marked points $v \in \mathbf{M}$ as follows: if $v$ is a puncture $C$, then $b_v = a_C$. If $v$ is a cilium residing on an even component $C$, then $b_v = \pm a_C$, with $\pm$ sign consistent with the black-white coloring on $C$. If $v$ is cilium on an odd component, we set $a_v = 0$.

Now consider any vertices $v_1,v_2$ forming an edge in the triangulation $T$. We claim  
\begin{equation}\label{balancedvertices} 
a_{v_1}+a_{v_2} = 0.
\end{equation} 

Indeed, if $v_1,v_2$ are the endpoints of an arc $\gamma \in T$, the $\gamma^{\text{th}}$ component of the relation \eqref{relationofrelations} is $a_{v_1}+a_{v_2}$ by construction, and \eqref{balancedvertices} holds.  If they are the endpoints of a boundary segment, then \eqref{balancedvertices} clearly holds.

However, in any given triangle in $T$ with vertices $v_1,v_2,v_2$, the only way for \eqref{balancedvertices} to hold for all $3$ of the pairs $(v_1,v_2),(v_1,v_3),(v_2,v_3)$ is if 
$a_{v_1} =a_{v_2} =a_{v_3} = 0 .$ 
Varying the vertex and triangle containing it, this establishes that all $a_v =0$ for all $v \in \mathbf{M}$, and hence all $a_C = 0$, as desired.    
\end{proof}

\section{Appendix: The starfish lemma on a nerve}\label{StarfishSecn}
We give the appropriate generalization of the Starfish Lemma \cite[Proposition 3.6]{tensors} from a star neighborhood to a nerve. Our proof follows the proof of the Starfish Lemma in \cite{CABook}, with appropriate modifications.

Let $R$ be a  domain. We say two elements $r,r' \in R$ are \emph{coprime} if they are not contained in the same prime ideal of height $1$. When $R$ is a unique factorization domain, every pair of non-associate irreducible elements are coprime. 

\begin{prop}\label{starfishprop}
Let $\mcn$ be a nerve in $\bbt_n$. Let $\mcr$ be a $\bbc$-algebra and a Noetherian normal domain. Let $\mce$ be a seed pattern of geometric type, satisfying the following: 
\begin{itemize}
\item all frozen variables are in $\mcr$
\item for each vertex $t \in \mcn$, the cluster $\mathbf{x}(t) \subset \mcr$, and the cluster variables $x \in \mathbf{x}(t)$ are pairwise coprime elements of $\mcr$; 
\item for each edge $t \xrightarrow{k} t'$ in $\mcn$, the cluster variables $x_k(t)$ and $x_k(t')$ are pairwise coprime. 
\end{itemize}
Then the cluster algebra $\mca$ defined by $\mce$ satisfies $\mca \subset \mcr$. 
\end{prop}

The proof relies on the following two lemmas, the first of which is a standard fact from commutative algebra.  For a prime ideal $P$, let $R_P = R[(R / P)^{-1}]$ denote the localization of $R$ away from $P$.   
\begin{lem}[{\cite[Theorem 11.5]{Matsumura}}]\label{heightoneprimes} For a normal Noetherian  domain $R$, the natural inclusion $R \subset \cap_{\textnormal{ht $P = 1$}} R_P$ (intersection over height one primes) is an equality. 
\end{lem}

\begin{lem}\label{goodclusterforP} With hypotheses as in Proposition \ref{starfishprop}, let $P$ be a height one prime ideal in $R$. Then at least one of the products 
\begin{equation}\label{hartogseqn}
\prod_{x \in \mathbf{x}(t), t \in \mcn}  x \end{equation}
is not in $P$.
\end{lem}

\begin{proof} By the coprimeness in each cluster $t \in \mcn$, at most one of the cluster variables $x$ in a product \eqref{hartogseqn} satisfies $x \in P$. We will show that for at least one $t$, \emph{none} of the cluster variables is in $P$, establishing our claim since $P$ is prime. Pick any vertex $t_0 \in \mcn$, and suppose the cluster variable $x_i \in P$. Given an edge $t_0 \xrightarrow{j} t_0' \subset \mcn$ where $j \neq i$, the cluster variable $x_j(t_0') \notin P$ by the coprimality assumption in the cluster at $t_0'$. Repeatedly applying this assumption while mutating along the nerve, using the connectedness hypothesis and the fact that every edge label shows up in the nerve, we finally arrive at a vertex $t \in \mcn$ such that the edge $t \xrightarrow{i} t' \subset \mcn$, and all of the extended cluster variables $x_j \in \tilde{\mathbf{x}}(t)$ with $j \neq i$ are not in $P$. By the coprimeness assumption along edge $i$, we see $x_i(t') \notin P$, and the cluster at $t'$ is one where the product  \eqref{hartogseqn} is not in $P$. 
\end{proof}

\begin{proof}[Proof of Proposition \ref{starfishprop}]
We need to prove each cluster variable $z$ is in $R$. By Lemma \ref{heightoneprimes}, it suffices to show $z \in R_P$ for any height one prime $P$. Indeed, by Lemma \ref{goodclusterforP} there is a cluster $t \in \mcn$ such that $\prod_{x \in \mathbf{x}(t)} x \notin P$. By the Laurent Phenomenon, $z$ is a Laurent polynomial in the elements of $\mathbf{x}(t)$, with coefficients in $\bbc[x_{n+1},\dots,x_{n+m}]$. In particular, $z \in R_P$, as desired. 
\end{proof}

\section{Appendix: Grassmannians and Band Matrices}\label{GrassmannianandBands}
As an illustration of Proposition \ref{qhnormalizedprop}, we extend the constructions in Example \ref{GrandBandsQH} and Example \ref{GrandBandsQI} from the case $(k,n) = (2,5)$ to general $(k,n)$. Let 
\begin{equation}\label{bfxdefn}
\bfx = \tGr(n-k,n)
\end{equation}
be the affine cone over the Grassmannian of $(n-k)$-dimensional subspaces of $\bbc^n$. Its points are the decomposable tensors in $\Lam^{n-k}{\bbc^n}$. 
Let 
\begin{equation}\label{bfydefn}
\bfy \cong \bbc^{(n-k)(k+1)}
\end{equation} 
be the affine space of $(n-k) \times n$ \emph{band matrices of width $k+1$}, i.e. the set of matrices $Y$ whose entries $y_{i,j}$ are zero unless $i \leq j \leq i+k$.  

We will describe  a quasi-isomorphism of the coordinate rings $\bbc[\bfx]$ and $\bbc[\bfy]$, and in particular a cluster structure on $\bbc[\bfy]$ which appears to be new. 

The coordinate ring $\bbc[\bfx]$ is the ring generated by the Pl\"ucker coordinates $\Delta_S$ as $S$ ranges over $(n-k)$-subsets of $n$. It has a  cluster structure of geometric type cf.~\cite{Scott} in which the frozen variables are those  Pl\"ucker coordinates consisting of cyclically consecutive columns. The non-frozen Pl\"ucker coordinates are all cluster variables. We will introduce a useful sign convention: if $S$ is any set of $(n-k)$ natural numbers, we let $\Delta_S$ denote the Pl\"ucker coordinate obtained by first reducing all the elements of $S$ to their least positive residue modulo $n$, sorting these residues, and then taking the corresponding Pl\"ucker coordinate. If there are fewer than $(n-k)$ distinct elements in $S$ modulo $n$, then $\Delta_S$ is identically zero.

The coordinate ring $\bbc[\bfy]$ contains minors $Y_{I,J}$ for $I \subset [1,n-k], J \subset [1,n]$ subsets of the same size denoting row and column indices respectively. It is a polynomial ring in the coordinate functions $Y_{i,j}$, $1 \leq i \leq j \leq i+k \leq n$. The following elements will serve as frozen variables in $\bbc[\bfy]$:  
\begin{align} 
Y_{1,1},Y_{2,2},\dots,Y_{n-k,n-k} \text{ and } Y_{1,k+1},Y_{2,k+1},\dots,Y_{n-k,n} \text{ , and} \\
Y_{[1,n-k],[2,n-k+1]},Y_{[1,n-k],[3,n-k+2]},\dots,Y_{[1,n-k],[k,n-1]};
\end{align}
we let $\ov{\bfp}$ denote the corresponding tropical semifield in these frozen variables. 

%These are exactly the irreducible factors of $F^*(\Delta_S)$ in $\bbc[\bfy]$, as $\Delta_S$ varies over the frozen variables in $\bbc[\bfx]$. 

Just as in Example \ref{GrandBandsQH}, there is a morphism of varieties $F \colon \bfy \to \bfx$ sending $Y \in \bfy$ to the decomposable tensor $Y[1] \wedge \cdots \wedge Y[n-k] \in \bfx$, where the $Y[i]$ are the rows of $Y$. The map on coordinate rings is 
\begin{equation}\label{Fstardefn}
F^* \colon \bbc[\bfx] \to \bbc[\bfy], \text{ defined by $F^*(\Delta_S) = Y_{[1,n-k],S}$.}
\end{equation}

Letting $\Delta_{S}$ be any non-frozen Pl\"ucker coordinate, one sees that 
\begin{equation}\label{conPluckers}
F^*(\Delta_S) = c(S) \cdot Y_{I(S),J(S)} 
\end{equation}
where $c(S) \in \ov{\bfp}$ and $Y_{I(S),J(S)}$ is a non-frozen irreducible row-solid minor in $\bbc[\bfy]$. The map $\Delta_S \mapsto Y_{I(S),J(S)}$ is a bijection between the non-frozen Pl\"ucker coordinates in $\bbc[\bfx]$ and the non-frozen irreducible row-solid minors in $\bbc[\bfy]$.  

Just as in Example \ref{GrandBandsQI}, there is a morphism of varieties $G \colon \bfx \to \bfy$ sending $X \in \bfx$ to the band matrix whose whose $(i,j)$ entry is a certain Pl\"ucker coordinate evaluated on $X$:
\begin{equation}\label{gijdefn}
G(X)_{i,j} = \Delta_{[i+k+1,n+i-1] \cup j }(X).
\end{equation}
Since the Pl\"ucker coordinate on the right hand side of \eqref{gijdefn} is only nonzero when \\ unless~$i \leq j \leq i+k$, $G(X)$ is indeed a point in $\bfy$. 
The map on coordinate rings is
\begin{equation}\label{Gstardefn}
G^*(Y_{i,j}) = \Delta_{[i+k+1,n+i-1] \cup j }.
\end{equation} 

\begin{thm}\label{GrandBandsThm}
Let $\bfx$ and $\bfy$ be the varieties in \eqref{bfxdefn} and \eqref{bfydefn}. Let $\ov{\mcf}_{>0}$ be the semfield of subtraction-free expressions in the $Y_{i,j}$. Abusing notation, let $F^* \colon \mcf_{>0} \to \ov{\mcf}_{>0}$ be the semfield map determined by the map $F^*$ from \eqref{Fstardefn}, and let $G^* \colon \ov{\mcf}_{>0} \to \mcf_{>0}$ be the map determined by \eqref{Gstardefn}. Then $\bbc[\bfy]$ is a cluster algebra of geometric type, and the maps $F^*$ and $G^*$ are quasi-inverses.
All of the irreducible row-solid minors in $\bbc[\bfy]$ are cluster or frozen variables. 
\end{thm}

\begin{proof}
The proof relies on the following well-known properties of the cluster structure on $\bbc[\bfx]$:  
\begin{enumerate}
\item There exist clusters in $\bbc[\bfx]$ consisting entirely of Pl\"ucker coordinates (called \emph{Pl\"ucker clusters}). Every non-frozen Pl\"ucker coordinate shows up in at least one of these clusters.  
\item The set of Pl\"ucker clusters are connected to one another by mutations whose exchange relations are \emph{short Pl\"ucker relations}, of the form $\Delta_{S \cup ik}\Delta_{Sj\ell} = \Delta_{Sij}\Delta_{Sk\ell}+ \Delta_{Sjk}\Delta_{Si\ell}$, for $i < j < k < l$ and $Sij$ denotes the union $S \cup \{i,j\}$).
\item For certain Pl\"ucker clusters, every neighboring cluster is again a Pl\"ucker cluster.  
\end{enumerate}

The first two of these facts are consequences of the technology of plabic graphs and square moves, the third  fact follows by considering a particular explicit Pl\"ucker seed for the Grassmannian (one whose quiver is a ``grid quiver''). 

We can now state Theorem \ref{GrandBandsQH} more carefully. Let $\Sigma$ be any Pl\"ucker cluster in $\bbc[\bfx]$. For each Pl\"ucker coordinate $x_i \in \Sigma$, let $\ov{x}_i$ be the irreducible row-solid minor in $\bbc[\bfy]$ related to $x_i$ as in \eqref{conPluckers}. We will see below that $\{\ov{x}_i \colon x_i \in \Sigma \}$ are algebraically independent generators for $\ov{\mcf}_{>0}$ over $\ov{\bfp}$. Assuming this has been proved, applying the construction in Proposition \ref{qhnormalizedprop}, we obtain a semifield map $c_\Sigma \colon \mcf_{>0} \to \ov{\bfp}$ and a normalized seed $F^*(\Sigma)$ whose cluster variables are $\ov{x}_i = \frac{F^*(x_i)}{c_\Sigma(x_i)}$ as in \eqref{seedhomtopatternhomeqsiv}. The semifield map $c_\Sigma$ satisfies $c_\Sigma(\Delta_S) = c(S)$ for all $S \in \Sigma$, where $c(S)$ is defined by \eqref{conPluckers}. The \emph{key claim} is that in fact $c_\Sigma(\Delta_S) = c(S)$ holds for all $\Delta_S$, and therefore the semifield map $c_\Sigma$ does not depend on $\Sigma$.

From this key claim, it follows that the seeds $F^*(\Sigma)$ are all related to each other by mutation using Proposition \ref{qhnormalizedprop}. Furthermore, every non-frozen irreducible row-solid minor is a cluster variable in $\ov{\mce}$ by \eqref{seedhomtopatternhomeqsiv}. Since this includes all of the $Y_{i,j}$, this shows that the $F^*(\Sigma)$ are indeed seeds -- each seed has the expected size necessary to form a transcendence basis for $\bbc(\bfy)$, the seeds are all related to each other by mutation, and their union clearly contains a generating set for the field of fractions. The cluster algebra for the resulting seed pattern clearly contains $\bbc[\bfy]$. The opposite containment follows from the Algebraic Hartogs' argument on a starfish cf.~Section~\ref{StarfishSecn}, using Fact (3) above. 

Thus it remains to check the key claim that $c_\Sigma(\Delta_S) = c(S)$ for all non-frozen $S$. By Fact (2), it suffices to check that $c$ preserves the short Pl\"ucker relations, i.e.
\begin{equation}\label{checkcsmatch}
c(Sik)c(Sj\ell) = c(S ij)c(S k\ell) \oplus  c(S jk)c(S i\ell).
\end{equation}

Verifying \eqref{checkcsmatch} is a direct piecewise check: the exponent of $Y_{a,a}$ in the left hand side of \eqref{checkcsmatch} is $0,1,2$ according to whether neither, one of, or both of $Sik$ and $Sj\ell$ contain the interval $[1,a]$. Performing the similar computation for $Sij$ and $Sk \ell$, as well as $Sjk$ and $Si \ell$, and taking the minimum of their respective answers, gives the exponent of $Y_{a,a}$ in the right hand side, and we claim these left and right hand exponents are always equal. This can be done by a case analysis: let $E$ be the largest number such that $[1,E] \subset Sijkl$. If $E <i$, then both sides return a $2$ if $a \leq E$, and $0$ otherwise. If $i \leq E < j$, then both sides return a $2$ if $a < i$, return a $1$ if $i \leq a \leq E$, and return a $0$ otherwise. If $j \leq E $ then both sides return a $2$ if $a < i$, return a $1$ if $i \leq a < j$, and return a $0$ otherwise. A similar calculation checks that the exponents of $Y_{a,a+k}$ match up in both sides of \eqref{checkcsmatch}.

Finally we check that $G^*$ is a quasi-inverse to $F^*$.  By Lemma \ref{constructQIs}, we only need to see that $G^*$ preserves coefficients and that $G^* \circ F^*$ is proportional to the identity. It suffices to check that  %(both frozen and non-frozen), 
$G^* \circ F^* (\Delta_S) \asymp \Delta_S$ for every $\Delta_S$. This follows from the determinantal identity Lemma  \ref{flattoband} below, applied to $G^* \circ F^* (\Delta_S) = G^* (Y_{[1,n-k],S})$. 
\end{proof}

\begin{lem}\label{flattoband}  Let $I = [a,a+s-1]$ be some consecutive subset of $[n-k]$, and $J $ a subset of $[a,a+s-1+k]$ of size $s$. Then for $X \in \bbc[\bfx]$, 
\begin{equation}\label{thirdminoreq} Y_{I,J}(G(X)) =  \left(\prod_{i=a}^{a+s-2} \Delta_{[i+k+1,n+i]  }(X) \right) \cdot  \Delta_{ [a+k+s,n+a-1] \cup J}(X)
\end{equation}
\end{lem}

Notice that the first product on the right hand side of \eqref{thirdminoreq} is a monomial in the frozen Pl\"ucker coordinates.

\begin{proof} Proceed by induction on $s$. It's clear when $s=1$. For $s >1$, we will need a Pl\"ucker relation 
\begin{equation}\label{LongPlucker}
\Delta_{[a+k+1,n+a]}(X) \Delta_{[a+k+s,n+a-1] \cup J}(X) = \sum_{\ell=1}^s (-1)^{\ell+1}\Delta_{[a+1+k,n+a-1] \cup j_\ell}(X) \Delta_{[a+k+s,n+a] \cup (J-j_\ell)}(X),
\end{equation}
see e.g. \cite[Section 9.1, Exercise 1]{YoungTableaux}. Let $J = \{j_1,\dots,j_s\}$ with $j_1 < j_2 < \dots < j_s$. Assuming \eqref{thirdminoreq} holds for smaller values of $s$, we expand along the first row to see
\begin{align*}
Y_{I,J}(G(X)) &= \sum_{\ell = 1}^s (-1)^{\ell+1} Y_{a,j_\ell}(G(X)) Y_{(I-a),(J-j_\ell)}(G(X)) \\
&= \sum_{\ell = 1}^s (-1)^{\ell+1} \Delta_{[a+k+1,n+a-1]}(X) (\prod_{i=a+1}^{a+s-2} \Delta_{i+k+1,n+i}(X)) \cdot \Delta_{[a+k+s,n+a] \cup (J-j_\ell)} \\
&= (\prod_{i=a+1}^{a+s-2} \Delta_{i+k+1,n+i}(X))  \sum_{\ell=1}^s (-1)^{\ell+1}\Delta_{[a+1+k,n+a-1] \cup j_\ell}(X) \Delta_{[a+k+s,n+a] \cup (J-j_\ell)}(X), 
\end{align*}
and the result follows using \eqref{LongPlucker}.   
\end{proof}

\begin{rmk} In the case $k=2$, our construction is the ``motivating example'' considerd by Yang and Zelevinsky \cite{YZ}. They establish that the homogeneous coordinate ring of a certain $\SL_{n+1}$-double Bruhat cell is a Dynkin type $A_n$ cluster algebra with principal coefficients. The elements of this double Bruhat cell are $(n+1) \times (n+1)$ band matrices of width $3$. Their example follows from ours by setting certain frozen variables equal to $1$, and setting $Y_{1,1} $ and $Y_{n-k,n}$ equal to $0$ (this latter operation is permissible because both of these frozen variables are isolated vertices in the quivers for $\bbc[\bfy]$).

We also remark that it is already known that the Grassmannian cluster algebras are quasi-isomorphic to a polynomial ring, by a fairly uninteresting quasi-isomorphism. Indeed, we can realize the affine space of $(n-k) \times k$ matrices as the closed subvariety of $\tGr(n-k,n)$ defined by specializing the frozen variable $\Delta_{[1,n-k]}$ to  $1$, and this specialization is a quasi-isomorphism. The resulting cluster structure on the polynomial ring is unrelated to the one we have given in this section. 
\end{rmk}

\end{document}